\documentclass[11pt,oneside]{amsart}
\usepackage{ifthen}
\usepackage{amsfonts}
\usepackage{amssymb}
\usepackage{epsf}
\usepackage{amsopn}
\usepackage{amsmath,amscd}
\usepackage{caption}
\usepackage{subcaption}

\theoremstyle{plain}
\newtheorem{thm}{Theorem}[section]
\newtheorem{prop}[thm]{Proposition}

\newtheorem{cor}[thm]{Corollary}
\newtheorem{lem}[thm]{Lemma}

\newtheorem{rem}[thm]{Remark}

\theoremstyle{definition}
\newtheorem{Def}[thm]{Definition}

\newtheorem{convention}[thm]{Convention}

\title{Time complexity of the conjugacy problem in relatively hyperbolic groups}

\author[I.~Bumagin]{Inna Bumagin}
      \address{ School of Mathematics and Statisics\\
               Carleton University \\
               Ottawa, ON, Canada  K1S 5B6}
      \email{bumagin@math.carleton.ca}
      
\thanks{Research supported by NSERC grant}

      \subjclass[2010]{ %2000 Classification!
      20F06, %Cancellation theory; application of van Kampen diagrams
      20F65,  % 	Geometric group theory
      20F67,   % Hyperbolic groups and nonpositively curved groups
      57M07,   %	Topological methods in group theory
      05C25   %	Graphs and abstract algebra (groups, rings, fields, etc.)
      }
      \keywords{Algorithmic problems, conjugacy problem, relatively hyperbolic groups, time complexity}

\begin{document}

%\date{version of: \today}

\begin{abstract}
If $u$ and $v$ are two conjugate elements of a hyperbolic group then the length of a shortest conjugating element for $u$ and $v$ can be bounded by a linear function of the sum of their lengths, as was proved by Lysenok in~\cite{Lysenok90}. Bridson and Haefliger showed in~\cite{BridsonHaefliger} that in a hyperbolic group the conjugacy problem can be solved in polynomial time. We extend these results to relatively hyperbolic groups. In particular, we show that both the conjugacy problem and the conjugacy search problem can be solved in polynomial time in a relatively hyperbolic group, whenever the corresponding problem can be solved in polynomial time in each parabolic subgroup. We also prove that if $u$ and $v$ are two conjugate hyperbolic elements of a relatively hyperbolic group then the length of a shortest conjugating element for $u$ and $v$ is linear in terms of their lengths.
\end{abstract}

\maketitle

\tableofcontents

\section{Introduction}
%\subsection{Conjugacy problem and Conjugacy Search problem}\label{sec:conjproblem}
The conjugacy problem is one of the three classical algorithmic problems formulated by Max Dehn in 1912. It asks the following. 

\textbf{The Dehn Conjugacy Problem.} Find an algorithm that takes as input a finite presentation of a group $G=\left\langle S\mid \mathcal{R} \right\rangle $ and two freely reduced products of the generators $u,v\in F(S)$, and decides whether or not the elements of $G$, defined by the given products of generators, are conjugate in $G$.

When stated in this generality, the conjugacy problem is known to be unsolvable, which was first shown by Novikov in 1954~\cite{NovikovConj}. The next year Novikov proved in ~\cite{Novikov} that the word problem is unsolvable for finitely presented groups in general, and two years later Boone published his proof in~\cite{Boone}. Therefore, to obtain positive results, one needs to restrict to subclasses of the class of finitely presented groups.

 Gromov stated in his seminal paper~\cite{Gromov} that the conjugacy problem is solvable in hyperbolic groups. The fundamental idea is that if the given elements $u$ and $v$ are conjugate in the group then the length of a shortest conjugating element can be bounded in terms of the lengths of $u$ and $v$. A rough bound on the length of a conjugating element is exponential, and a straightforward algorithm to find a conjugating element (or to verify that the given elements are not conjugate) is double exponential. Later on, a linear bound on the shortest length of a conjugating element was found by Lysenok~\cite[Lemma 10]{Lysenok90} (see also Lemma~\ref{lem:linearbound} in the present paper); this bound leads to an algorithm with exponential time complexity. Better bounds and a much more elegant, polynomial time, algorithm for hyperbolic groups can be found in the book by Bridson and Haefliger~\cite{BridsonHaefliger}; this algorithm is cubic-time. This result was improved by Bridson and Howie in~\cite{BridsonHowie} and then improved even further by Epstein and Holt, who found a linear time algorithm ~\cite{EpsteinHolt}.
 
Closely related to the conjugacy problem is the following question.
 
\textbf{The Conjugacy Search Problem.} Find an algorithm that takes as input a finite presentation of a group $G=\left\langle S\mid \mathcal{R} \right\rangle $ and two freely reduced products of the generators $w_u,w_v\in F(S)$ that define elements $u$ and $v$ conjugate in $G$, and finds a conjugating element $g\in G$ for $u$ and $v$. 
 
Clearly, in a countable group with solvable word problem the conjugacy search problem can be solved by enumerating the group elements $g_1,g_2,\dots$ and deciding for every $i=1,2,\dots$ whether or not the equality $g_iug_i^{-1}v^{-1}=1$ holds in $G$. A conjugating element will be found, eventually. In a hyperbolic group, due to the linear bound on the length of a conjugating element, the straightforward procedure is exponential time. 

Both the conjugacy problem and the conjugacy search problem can be stated in the language of equations over groups. Namely, the conjugacy problem asks to find out whether the orientable quadratic equation $xux^{-1}=v$ has a solution in $G$. The conjugacy search problem asks to find a solution, when we know that it exists. The estimates obtained in~\cite{Lysenok90} and in~\cite{BridsonHaefliger}  for the length of a conjugating element can be compared with the estimate of $O(n^4)$ on the length of a minimal solution to a system of orientable quadratic equations over a torsion-free hyperbolic group, obtained by Mohajeri in her Ph.D. thesis~\cite{Mohajeri}, and with the linear estimates for the length of solutions to quadratic equations in free groups, due to Kharlampovich and Vdovina~\cite{KVLinearEst}.  At the same time, one notes the striking difference between the time complexity obtained in~\cite{EpsteinHolt} and the result due to Kharlampovich, Mohajeri, Taam and Vdovina~\cite{KMTV} that solving quadratic equations in hyperbolic groups is NP-hard.

% % % % % % % % % % % % % % % %
We consider the class of relatively hyperbolic groups, defined in Section~\ref{sec:defnsnotation}. In this class, a reasonable question is as follows.

\textbf{The Conjugacy Problem for Relatively Hyperbolic Groups.} Find an algorithm that takes as input 
\begin{itemize}
\item a finite relative presentation of a group $G=\left\langle S_0,P_1,P_2,\dots,P_m\mid \mathcal{R} \right\rangle $, hyperbolic relative to the set of subgroups $\mathcal{P}=\{P_1,P_2,\dots,P_m\}$, and finite generating sets $S_1,S_2,\dots, S_m$ such that $P_i=\left\langle S_i\right\rangle,\ \forall i $,
\item solutions to the word problem and to the conjugacy problem in each one of the subgroups in $\mathcal{P}$ and 
\item two freely reduced products of the generators $w_u,w_v\in F(S)$, where $S=\cup_{i=0}^m S_i$,
\end{itemize}  
 and decides whether or not the elements $u,v\in G$, defined by the products $w_u$ and $w_v$, are conjugate in $G$.

Solvability of the conjugacy problem in relatively hyperbolic groups, also declared by Gromov in~\cite{Gromov}, was proved by the author in~\cite{BumaginConj},~ and also by Osin in~\cite{OsinRelHypGps} for hyperbolic elements. Both papers provide estimates that lead to inefficient algorithms. Better estimates on the length of conjugating elements in relatively hyperbolic groups were obtained by Ji, Ogle and Ramsey, who gave a polynomial bound on the length in~\cite{JiOgleRamsey}. The degree of the polynomial bound was improved considerably by O'Connor in~\cite{OConnor}.

A key to obtaining a polynomial time algorithm to solve the conjugacy problem is an observation due to Bridson and Haefliger~\cite{BridsonHaefliger} that we generalize in Theorem~\ref{thm:upperboundconst}(1). Essentially, it says the following. One can compute shorter representatives $\bar{w}_u\in[u]_G$ and $\bar{w}_v\in[v]_G$ of the conjugacy classes of the given elements $u$ and $v$ and obtain a bound on the length of an element $g$ conjugating $\bar{w}_u$ and  $\bar{w}_v$. In particular, if $u$ and $v$ are hyperbolic elements (see Definition~\ref{def:parabolicelt}) conjugate in $G$, then the length of $g$ will be bounded by a universal constant which is independent of $u$ and $v$. As a consequence of Theorem~\ref{thm:upperboundconst}(1) and Proposition~\ref{cor:Lemma431}, we obtain the following linear estimate on the length of a conjugating element.
\begin{thm}\label{thm:IntroEstimate} Let $G$ be a group, hyperbolic relative to the set of subgroups $\mathcal{P}=\{P_1,\dots,P_m\}$. Let $S=\cup_{i=0}^{m} S_i$ be a finite generating set for $G$ such that $P_i=\left\langle S_i\right\rangle $ for $i=1,2,\dots,m$. Let $F(S)$ be the free group on $S$. Let $u,v\in G$ be defined by products of generators $w_u,w_v\in F(S)$. Suppose that $u$ and $v$ are hyperbolic elements conjugate in $G$. Then there is a conjugating element $g\in G$ such that $|g|_S\leq N(\delta)(|w_u|_S+|w_v|_S)+M(\delta)$, where $N(\delta)$ and $M(\delta)$ are constants, computable from the given presentation of $G$.
\end{thm}

If $u$ and $v$ are parabolic elements (see Definition~\ref{def:parabolicelt}) conjugate in $G$ then estimates on the length of a shortest conjugating element will depend on the corresponding estimates in the parabolic subgroups of $G$. Note that parabolic subgroups of $G$ do not have to belong to the same class as abstract groups. A  general statement is given in Theorem~\ref{thm:upperboundconst}(2). For some classes of parabolic subgroups our estimates can be made more precise, as we discuss below.

The following theorem refers to a particular case of relatively hyperbolic groups with abelian parabolic subgroups. These groups have been attracting a great deal of attention lately; we just mention some examples. For instance, $\mathbb{R}^n$-free groups are hyperbolic relative to abelian subgroups, which was proved by Guirardel in~\cite{Guirardel}, using a theorem proved by Dahmani in~\cite{Dahmani}. Also, finitely generated groups acting properly and cocompactly on $CAT(0)$ spaces with isolated flats are hyperbolic relative to abelian subgroups; we refer the interested reader to Hruska's thesis~\cite{HruskaThesis} and to the paper by Hruska and Kleiner~\cite{HruskaKleiner} for details. It follows from the results of Kharlampovich and the author~\cite{BumaginKharlampovich} that $\mathbb{Z}^n$-free groups belong to the class defined by Hruska. An important subclass of $\mathbb{Z}^n$-free groups consists of finitely generated fully residually free groups, also known as limit groups introduced by Sela, who also proved the equivalence of the two definitions in~\cite{SelaMRDiagrams01}; these groups were shown to be  $\mathbb{Z}^n$-free by Kharlampovich and Miasnikov in~\cite{KharlampovichMyasnikov98}. Alternatively, the above inclusions for limit groups follow from the combination theorems for relatively hyperbolic groups proved by Alibegovi\v{c} in~\cite{Alibegovic} and by Dahmani in~\cite{Dahmani}, and from the theorem proved by Alibegovi\v{c} and Bestvina in~\cite{AlibegovicBestvina}.

If $G$ is a group hyperbolic relative to abelian subgroups then, according to Theorem~\ref{thm:EstimateAbelian} below, a linear bound on the length of a conjugating element applies to arbitrary $u,v\in G$. Theorem~\ref{thm:EstimateAbelian} follows from Theorem~\ref{thm:upperboundconst} and Remark~\ref{rem:conjparabolic}.
\begin{thm}\label{thm:EstimateAbelian}
Let $G$ be a finitely generated group, hyperbolic relative to the set of abelian subgroups $\mathcal{P}=\{P_1,\dots,P_m\}$. Let $S=\cup_{i=0}^{m} S_i$ be a finite generating set for $G$ such that $P_i=\left\langle S_i\right\rangle $ for $i=1,2,\dots,m$. Let $F(S)$ be the free group on $S$. Let $u,v\in G$ be defined by products of generators $w_u,w_v\in F(S)$. If $u$ and $v$ are conjugate in $G$ then there is a conjugating element $g\in G$ such that $|g|_S\leq N(\delta)(|w_u|_S+|w_v|_S)+M_{ab}(\delta)$, where $N(\delta)$ and $M_{ab}(\delta)$ are constants, computable from the given presentation of $G$.
\end{thm}
Our solution to the conjugacy problem uses Farb's solution to the word problem in relatively hyperbolic groups.
\begin{thm} \label{thm:Farb} \cite[Theorem 3.7]{Farb}
Suppose $G$ is a group hyperbolic relative to the set of subgroups $\mathcal{P}=\{P_1,P_2,\dots,P_m \}$, defined by a Dehn presentation, and suppose $P_i$ has word problem solvable in time $O(f(n))$, $\forall i$. Then there is a curve-shortening algorithm that gives an $O(f(n)\log n)$-time solution to the word problem for $G$.
\end{thm}
The curve-shortening algorithm is designed to obtain a relative $k$-local geodesic representing a given element of $G$. We sketch Farb's argument in Lemma~\ref{lem:Farb} and in Proposition~\ref{cor:Lemma431} adapt it to obtain cyclic relative $k$-local geodesics.

The conditions of Theorem~\ref{thm:Farb} could be relaxed to allow the group $G$ be defined by an arbitrary finite presentation; the word problem in $G$ will remain solvable. This is due to Dahmani~\cite[Theorem 0.1]{DahmaniFindingRelHypStructures}, who gave an algorithm to compute a relative Dehn presentation for $G$ and a factor $N$ for a relative linear isoperimetric inequality for $G$ from a finite presentation for $G$, generating sets $S_1,S_2,\dots,S_m$ for the parabolic subgroups and solution to the word problem in the parabolic subgroups (see Theorem~\ref{prop:precomputation} for details). Then one can find a hyperbolicity constant $\delta$, also used in Farb's curve-shortening algorithm.  However, the complexity of those additional computations exceeds by far the complexity of the  solution to the word problem in $G$. This is why we also assume that $G$ is defined by a relative Dehn presentation when discussing the complexity of our algorithms. Moreover, in Section~\ref{sec:Preliminary} we explain in detail what additional data, that does not depend on the given elements $u$ and $v$ and is used in our algorithms, is considered to be computed in advance.

The following theorem shows that we do not need solution to the conjugacy problem in parabolic subgroups to solve the conjugacy problem for hyperbolic elements of $G$ and even to find a conjugating element. Moreover, we do not need to know \textit{a priori} whether given elements are hyperbolic, because we can determine this using only solution to the word problem. The statement also provides estimates for the time complexity of these computations.
\begin{thm}\label{thm:IntroConj} Let $G$ be a group, hyperbolic relative to the set of subgroups $\mathcal{P}=\{P_1,\dots,P_m\}$. The following data is considered input of our algorithms.
\begin{enumerate}
\item[(i)] A relative Dehn presentation $\left\langle S_0,P_1,\dots,P_m \mid \mathcal{R} \right\rangle $ for $G$, along with finite generating sets $S_i$ for parabolic subgroups: $P_i=\left\langle S_i\right\rangle $ for $i=1,2,\dots,m$. We denote by $S$ the finite generating set $S=\cup_{i=0}^{m} S_i$ for $G$;
\item[(ii)] Solution(s) to the word problem in the parabolic subgroups; let $O(C_w^{(par)}(n))$ denote the (maximum) complexity of these procedures;
\item[(iii)]  Two words $w_u, w_v\in F(S)$, where $F(S)$ is the free group on $S$. We consider the maximum length $\bar{L}=\max\{|w_u|_S,|w_v|_S \}$ of these words the length of the input. Denote by $u$ and $v$ the elements of $G$, defined by the words $w_u$ and $w_v$, respectively.
\end{enumerate}
There are algorithms as follows:
\begin{enumerate}
\item  An algorithm that decides whether $u$ (or $v$) is a hyperbolic or a parabolic element of $G$. The time complexity of the algorithm is $O(\bar{L}C_w^{(par)}(\bar{L}))$.
\item \label{ItemConjSearch} For $u$ and $v$ hyperbolic in $G$, an algorithm that decides whether or not $u$ and $v$ are conjugate in $G$. Moreover, if $u$ and $v$ are conjugate then a conjugating element will be found. The time complexity of the algorithm is $O(\bar{L}^2 C_w^{(par)}(\bar{L}) \log \bar{L})$. 
\end{enumerate}
\end{thm}
Theorem \ref{thm:IntroConj}(1) follows from Theorem~\ref{thm:parabolic}, and Theorem \ref{thm:IntroConj}(2) follows from Theorem~\ref{thm:algorithm}(2) and Theorem~\ref{thm:Search}(1).  The statement of Theorem~\ref{thm:algorithm}(3), not included in Theorem~\ref{thm:IntroConj}, provides solution to the conjugacy problem for parabolic elements in $G$, and this is where we use solution(s) to the conjugacy problem in parabolic subgroups. Let $O(C_c^{(par)}(n))$ denote the (maximum) complexity of these solutions.
As an immediate corollary of Theorem~\ref{thm:algorithm}, we have the following estimate for the time complexity of the conjugacy problem.
\begin{thm} \label{thm:ComplexityConj} There is an algorithm which takes as input all of the data listed in Theorem~\ref{thm:IntroConj} (cf. Convention~\ref{convention51}) and solution to the conjugacy problem in each one of the parabolic subgroups of $G$, and decides whether or not $u$ and $v$ are conjugate in $G$. The time complexity of the algorithm is 
$$T_c(\bar{L})=\max\{O(C^{(par)}_c(\bar{L})),O(\bar{L}^2C^{(par)}_w(\bar{L})\log \bar{L})\}.$$
\end{thm}

Note that the algorithm from the statement of Theorem~\ref{thm:IntroConj}(\ref{ItemConjSearch}) solves also the conjugacy search problem for hyperbolic elements of $G$. 
 According to Theorem~\ref{thm:IntroConj} (see also  Theorem~\ref{thm:parabolic}), one does not have to specify whether or not $u$ and $v$ are hyperbolic elements. If the elements $u$ and $v$ happen to be hyperbolic in $G$ then a conjugating element will be found in polynomial time, whenever the word problem in parabolic subgroups has polynomial time solution. This result applies to hyperbolic groups (where the parabolic subgroups are all trivial) and should be contrasted with the straightforward exponential time algorithm to solve the conjugacy search problem in a hyperbolic group, mentioned earlier. 
 The algorithm that we present in Theorem~\ref{thm:Search} solves the conjugacy search problem for all the elements of $G$. However, if $u$ and $v$ are parabolic then the running time $T_{search}(\bar{L})$ of the algorithm will depend on the complexity $O(C_{search}^{(par)}(n))$ of solution to the conjugacy search problem in parabolic subgroups. Theorem~\ref{thm:conjsearch} provides the following estimate:
 $$T_{search}(\bar{L})=\max\{T_c(\bar{L}),O(C^{(par)}_{search}(\bar{L}))\}.$$ 
 Here $T_c(\bar{L})$ is the complexity of the conjugacy problem in $G$, see Theorem~\ref{thm:ComplexityConj}.

As an application of the results obtained in the paper, we prove the following theorem.
\begin{thm}\label{thm:abeliansolvable}
Let $G$ be a finitely generated group hyperbolic relative to subgroups $P_1,\dots,P_m$. Then the word problem, the conjugacy problem and the conjugacy search problem in $G$ can be solved in polynomial time if the parabolic subgroups are abelian, free solvable or Artin groups of extra-large type. More precisely, we have the following.
\begin{enumerate}
\item If the parabolic subgroups $P_1,\dots,P_m$ are free solvable then the time complexity of the word problem is $O(n^3\log n)$. The time complexity of the conjugacy and of the conjugacy search problem is $O(n^5\log n)$ for hyperbolic elements and $O(n^8)$ for parabolic elements of $G$. 
\item If the parabolic subgroups $P_1,\dots,P_m$ are Artin groups of extra-large type then the time complexity of the word problem is $O(n^2\log n)$. The time complexity of the conjugacy and of the conjugacy search problem is $O(n^4\log n)$ for hyperbolic elements and $O(n^3)$ for parabolic elements of $G$. 
\item If the parabolic subgroups $P_1,\dots,P_m$ are abelian then the time complexity of the word problem is $O(n\log n)$, and the time complexity of the conjugacy and of the conjugacy search problem is $O(n^3\log n)$.
\end{enumerate}
\end{thm}
\begin{proof} This is a consequence of Theorem~\ref{thm:algorithm}, proved in Section~\ref{sec:Alg}. The assertion (1) follows from the results of Myasnikov, Roman'kov, Ushakov and Vershik~\cite{MRUV10} who showed that the word problem in free solvable groups can be solved in cubic time, and from the algorithm due to Vassileva~\cite{Vassileva11} solving the conjugacy problem in free solvable groups in $O(n^8)$. In the assertion (2), the complexity of the word problem follows from the theorem by Peifer in~\cite{Peifer} stating that the Artin groups of extra-large type are bi-automatic, and the complexity of the conjugacy problem for parabolic elements is a result by Holt and Rees in~\cite{HoltRees}. The complexity of the word problem in the assertion (3) is a particular  case of Theorem~\ref{thm:Farb}.
\end{proof}
In a recent preprint~\cite{AntolinCiobanu}, Antolin and Ciobanu provide a cubic-time algorithm for solving the conjugacy problem in groups hyperbolic relative to abelian subgroups.

Another application was suggested by Ashot Minasyan. In~\cite{RipsConstruction} Rips gave an example of a finitely generated normal subgroup $K$ of a hyperbolic group $H$ such that $K$ is not hyperbolic as an abstract group. Moreover, the membership problem is not solvable for $K$ in $H$. The proof of the following theorem shows that $K$ is not a relatively hyperbolic group. Note that $\partial H$ is the compact space on which $K$
acts as a convergence group. Therefore, Theorem~\ref{thm:Ashot} gives an example of a finitely generated convergence group which is not relatively hyperbolic with respect to any
family of proper subgroups. Moreover, since $K$ is normal in $H$, its limit set coincides with the entire
Gromov boundary $\partial H$ of $H$. This can be contrasted with the characterization of relatively hyperbolic groups as geometrically finite convergence groups by Asli Yaman in~\cite{Yaman}. 

\begin{thm} (Minasyan)\label{thm:Ashot}
There is a finitely generated subgroup of a hyperbolic group that is not hyperbolic relative to any finite set of its own proper subgroups.
\end{thm}
\begin{proof} Let $K$ be the subgroup of the hyperbolic group $H$, both as constructed in~\cite[Theorem]{RipsConstruction}. Let $g\in H$ be arbitrary, we want to decide whether $g\in K$. 

Suppose that $K$ is hyperbolic relative to a set $\mathcal{P}=\{P_1,\dots, P_m \}$ of its proper finitely generated subgroups. Since $K$ and all of $P_i\in\mathcal{P}$ are subgroups of a hyperbolic group, the word problem is solvable in $K$ and in $P_i,\ \forall i$. Therefore, Theorem~\ref{thm:IntroConj} applies.  Enumerating elements of $K$ as words in the generators and applying Theorem~\ref{thm:IntroConj}(1) to each one of them, we can find a hyperbolic element $x\in K$ of infinite order (see Definition~\ref{def:parabolicelt} below). Since $K\lhd H$, $y=gxg^{-1}\in K$, and we can apply Theorem~\ref{thm:IntroConj}(1) to determine whether $y$ is parabolic or hyperbolic in $K$. If $y$ is a parabolic element of $K$ then it cannot be conjugate to $x$ in $K$, and if $y$ is hyperbolic in $K$ then by Theorem~\ref{thm:IntroConj}(2), we can decide whether $x$ and $y$ are conjugate in $K$. If not, then clearly, $g\notin K$. Otherwise, there is $k\in K$ such that $y=kxk^{-1}$, and $g\in K$ if and only if $gk^{-1}\in C_K(x)$. The proof of~\cite[Proposition 3.3]{MartinoMinasyan} shows that the latter inclusion can be decided effectively. 
Alternatively, the case when $x$ and $y$ are conjugate in $K$ could be handled as follows. We choose a f.p. torsion-free group $P$ with
unsolvable word problem, and apply Rips' construction to it. By the construction, $H$ is a
torsion-free hyperbolic group. Therefore, the
centralizer of the element $x$ is cyclic, and since $x \in K$ and $H/K$ is
torsion-free, it follows that $C_H(x)$ must be completely contained in $K$. 
Hence, $x$ and $y$ are conjugate in $K$ iff $g \in K$. The latter is
undecidable as the word problem in $P \cong H/K$ is unsolvable.
Thus, we can decide, whether or not $g\in K$, which contradicts~\cite[Corollary]{RipsConstruction}. Therefore, $K$ is not hyperbolic relative to the set $\mathcal{P}=\{P_1,\dots, P_m \}$.

\end{proof}

The paper is organized as follows. In Section~\ref{sec:defnsnotation} we collect the definitions and introduce the notation used throughout the paper. In Section~\ref{HypGpsEstimates} we establish the estimates on the length of a conjugating element, for both hyperbolic and parabolic elements. In Section~\ref{sec:Preliminary} we discuss preliminary computations. The algorithms are presented in Section~\ref{sec:Alg}.

\subsection*{Acknowledgment} Parts of this paper were written during my visit to the Research Centre Erwin Schr\"{o}dinger International Institute for Mathematical Physics (ESI) of the University of Vienna. I am grateful to the organizers of the research program ``Computation in groups'' Goulnara Arzhantseva, Olga Kharlampovich and Alexey Miasnikov for their kind invitation and to the ESI for the support.

\section{Definitions and Notation} \label{sec:defnsnotation} 

Let $G$ denote a group, generated by a finite set $S$.  We denote by $F(S)$ the free group with basis $S$. Let $\Gamma=\Gamma(G;S)$ be the Cayley graph of $G$ with respect to $S$. We assign length 1 to each edge of $\Gamma$ and denote by $d_{\Gamma}$ the obtained metric on $\Gamma$.

Let $I=[0,s]$ be an interval,  $\gamma: I\rightarrow\Gamma$ be an arbitrary path in $\Gamma$, and let $A$ and $B$ be two points in the image $im(\gamma)$ of $\gamma$. If $t_a$ and $t_b$ in $I$ are such that $\gamma(t_a)=A$ and $\gamma(t_b)=B$ then we denote by $l_{\gamma}(A,B)$ the length of the image of the subpath of $im(\gamma)$ joining $A$ and $B$: $l_{\gamma}(A,B)=l(im(\gamma|_{[t_a,t_b]}))$. If $t_a=0$ and $t_b=s$ are the enpoints of $I$, then we write $l_{\gamma}$ instead of $l_{\gamma}(A,B)$. A path $\gamma$ is called \textit{geodesic} if $l_{\gamma}(A,B)=d_{\Gamma}(A,B)$, for any two points $A,B\in im(\gamma)$. In what follows, abusing notation, we denote by $\gamma$ both the path and its image in $\Gamma$. We denote by $\gamma_-$ the origin $\gamma(0)$ and by $\gamma_+$ the terminus $\gamma(s)$ of the path $\gamma: [0,s]\rightarrow\Gamma$.

Let $V_g$ denote the unique vertex in $\Gamma$ that corresponds to $g$, and let $V_1\in\Gamma$ correspond to the identity of $G$. We denote by $|g|_{\Gamma}=d_{\Gamma}(V_1,V_g)$ the length of a geodesic path $\gamma\subset\Gamma$ representing $g$.

\begin{Def} \label{def:localgeodesic} Let $k>1$ be an integer. A path $\gamma$ is called a $k$-\textit{local geodesic} if every subpath of $\gamma$ of length at most $k$ is a geodesic. We say that $\gamma$ is a \textit{cyclic} ($k$-\textit{local}) \textit{geodesic} if  the concatenation $\gamma\circ\gamma$ is a ($k$-local) geodesic.
\end{Def}

\begin{Def}\label{def:hypspace} Let $\delta\geq 0$ be a real number. A geodesic metric space  $(X,d)$ is called $\delta$-\textit{hyperbolic} if for every geodesic triangle in $X$, each side of the triangle is contained in the $\delta$-neighborhood of the union of the other two sides.  A geodesic metric space  $(X,d)$ is called \textit{hyperbolic} if it is $\delta$-hyperbolic for some $\delta\geq 0$.
\end{Def}

\subsection{Relatively hyperbolic groups}\label{sec:defrelhyp}  In this section we give basic definitions and state some facts about the properties of relatively hyperbolic groups. We refer an interested reader to the papers by Farb~\cite{Farb}, Osin~\cite{OsinRelHypGps}, Bowditch~\cite{BowditchRelHypGps}, Hruska~\cite{HruskaRelHypQConvexity}, and the author~\cite{BumaginDefns} for various definitions and more detailed discussion of this class of groups.

Let $G$, $S$ and  $\Gamma=\Gamma(G;S)$ be as above. Let $P_1,P_2,\dots,P_m$ be finitely generated subgroups of the group $G$ and suppose that $P_i=\left\langle S_i\right\rangle $ for $i=1,2,\dots,m$, such that $S=S_0\cup S_1\cup S_2\cup\dots\cup S_m$. We assume that the sets $S_i$ $(i=0,1,\dots,m)$ are all finite. Denote by  $\Gamma_i=\Gamma(P_i;S_i)$ the Cayley graph of $P_i$ with respect to the generating set $S_i$ and consider the following presentation for $G$: 
\[
G=\left\langle S_0,P_1,\dots,P_m\mid R=1,\ R\in\mathcal{R}\right\rangle,
\]
where $\mathcal{R}$ is such that $G\cong F(S_0)*P_1*\dots *P_m/\left\langle \left\langle \mathcal{R}\right\rangle \right\rangle $. Note that in the latter presentation for $G$ the generating set $\hat{S}=S_0\cup P_1\dots\cup P_m$ is infinite if some of the subgroups $P_i$ are infinite. We denote by $\hat{\Gamma}=\hat{\Gamma}(G;\hat{S})$ the Cayley graph of $G$ with respect to $\hat{S}$. The graph $\hat{\Gamma}$ can be obtained from $\Gamma(G;S)$ as follows: add an edge between every pair of vertices in each $\Gamma_i$ and extend equivariantly with respect to the action of $G$. That is, add an edge between every pair of vertices in each left coset of $P_i$, for all $i=1,2,\dots,m$. Alternatively, one can add a vertex $v(gP_i)$ joined by an edge of length $1/2$ to every element in $gP_i$ for each left coset of $P_i$, for all $i$; see~\cite{Farb} for details. We call $\hat{\Gamma}$ a \textit{coned-off Cayley graph} and denote by $d_{\hat{\Gamma}}$ the (relative) metric associated with it. (Quasi-) geodesic paths in $\hat{\Gamma}$ are called \textit{relative (quasi-)geodesics}.

\begin{Def}\label{def:relhypgp}
The group $G$ is called \textit{weakly hyperbolic relative to subgroups} $\{P_1,\dots,P_m\}$ if the coned-off Cayley graph $\hat{\Gamma}$ is a hyperbolic metric space.
\end{Def}

Examples of weakly relatively hyperbolic groups, that can be found in~\cite{OsinWeaklyHyp} and in~\cite{BumaginConj}, show that in general, weakly relatively hyperbolic groups do not possess particularly nice properties. The problem is as follows. While a hyperbolic space can be characterized by the property that $(\lambda,\epsilon)$-quasi-geodesics with common endpoints are uniformly Hausdorff-close to one another, in a Cayley graph of a weakly relatively hyperbolic group this property no longer holds. We need to gain some control over quasi-geodesics to be able to draw interesting algebraic consequences.  This is why Farb~\cite{Farb} introduced an additional property, which he called the Bounded Coset Penetration (or BCP) property, see Definition~\ref{def:BCP}. We need to introduce some more terminology to explain it.

\begin{Def} \label{def:parabolicelt} An element $a\in G$ is called a \textit{parabolic element} if $a$ is conjugate into a subgroup $P_i\in\mathcal{P}$. Otherwise, $a$ is called a \textit{hyperbolic element}.
\end{Def}
In what follows, we distinguish two types of parabolic elements. Namely, if $a$ is written as a word in $S_i$, which we sometimes write as $a\in F(S_i)$ with slight abuse of notation, then clearly, $a\in P_i$. Parabolic elements of the other type are written as hyperbolic words and so are not ``obviously'' parabolic.
\begin{Def} \label{def:paraboliccomponent}
Let $h\in G$, and let $\alpha_h$ be the path in $\Gamma$ labelled by $h$, so that $h=lab(\alpha_h)$. A nonempty subword $p$ of $h$ is called a \textit{parabolic component of} $h$ if $p$ is an element of $P_i\in\mathcal{P}$, written as a word in $S_i$, and is a maximal parabolic subword of $h$ with respect to inclusion. If $h=h_1ph_2$, where $p\in P_i$ and $|p|_{\Gamma}\geq 1$ then we say that the path $\alpha_h$ in $\Gamma$ labelled by $h$ \textit{penetrates the coset} $h_1P_i$ \textit{along} $p$. We denote by $p_-$ the vertex of $\Gamma$ where $\alpha_h$ first enters $h_1P_i$ and by $p_+$ the vertex of $\Gamma$ where $\alpha_h$ last exits $h_1P_i$. The path $\alpha_p$ joining $p_-$ and $p_+$ inside the coset is a parabolic component of $\alpha_h$. We always assume that $\alpha_p$ is a geodesic path in $\Gamma_i$.
\end{Def}
Note that the relative length of each parabolic component of a path $\alpha$ equals 1.
\begin{Def} \label{def:connectedcomponents} Let $\alpha$ be a path in $\Gamma$, and let $gP_i$ be a coset of a parabolic subgroup of $G$. We say that $\alpha$ \textit{backtracks to} $gP_i$ if $\alpha$  joins two points in $gP_i$ but $\alpha$ is labeled by a non-parabolic word. A path $\alpha$ \textit{backtracks} if there are $g\in G$ and $i\in\{1,2,\dots,m\}$ such that $\alpha$ backtracks to $gP_i$. We say that $\beta$ is a \textit{path without backtracking} if no subpath of $\beta$ backtracks. 

Two distinct parabolic components $p$ and $q$ of a path $\alpha$ are \textit{connected} if $\alpha$  penetrates a coset $gP_i$ along $p$, exits $gP_i$, then backtracks to it and penetrates $gP_i$ along $q$. 
A parabolic component of a path $\alpha$ is called \textit{ isolated} if it is not connected to any other parabolic component of $\alpha$. In particular, if $\alpha$ is a path without backtracking then all the parabolic components of $\alpha$ are isolated.
\end{Def}
It can be readily seen that the following definition is equivalent to Farb's definition of bounded coset penetration \cite[Section 3.3]{Farb}.
\begin{Def}\label{def:BCP}
The group $G$ is said to satisfy the \textit{Bounded Coset Penetration (or BCP) property} if there is a constant $C=C(\lambda,\epsilon)$ such that the following condition holds. Let $\alpha$ and $\beta$ be  $(\lambda,\epsilon)$-quasi-geodesic paths without backtracking with common endpoints and distinct images in $\hat{\Gamma}$.  If $p$ is an isolated component of the closed path $\alpha\circ\beta^{-1}$ then $l_{\Gamma}(p)\leq C$.

We omit $\epsilon$ if $\epsilon=0$ and write $C(\lambda)$ instead of $C(\lambda, 0)$.
\end{Def}

\subsection{Notation} \label{sec:notation}
Let $u$ and $v$ be two elements of $G$, given as products of generators, $w_u$ and $w_v$ in $F(S)$, where $F(S)$ is the free group on $S$. Let $\bar{L}=\max\{ |w_u|_S,|w_v|_S \}$ be the maximum length of $w_u$ and $w_v$ in the word metric of $F(S)$; the length $\bar{L}$ is considered the length of the input of our algorithms. Throughout the paper, we denote by $\gamma_u$ and $\gamma_v$ two relative cyclic geodesics and by $\alpha$ and $\beta$ two relative cyclic $(8\delta+1)$-local geodesics that represent the elements $u$ and $v$, correspondingly. We denote by $L_g=\max\{l_{\gamma_u},l_{\gamma_v}\}$ the maximum relative lehgth of $\gamma_u$ and $\gamma_v$, and by $L=\max\{l_{\alpha},l_{\beta} \}$ the maximum relative lehgth of $\alpha$ and $\beta$. We obtain $\alpha$ and $\beta$ by shortening the paths labelled by $w_u$ and $w_v$ (see Proposition~\ref{cor:Lemma431} for details); therefore, we shall always assume that $L\leq \bar{L}$.

Suppose that $u$ and $v$ are conjugate in $G$, so that $v=gug^{-1}$ for some $g\in G$.  We denote by $\sigma$ and $\tau$ two relative geodesics representing $g$, such that  the concatenation $\sigma\circ\gamma_u\circ\tau^{-1}\circ\gamma_v^{-1}$ is a geodesic quadrilateral $Q_g$  and the concatenation $\sigma\circ\alpha\circ\tau^{-1}\circ\beta^{-1}$ is a (quasi-geodesic) quadrilateral $Q$, both are closed paths at the identity in the coned-off Cayley graph $\hat\Gamma$, see Figure~\ref{fig:F1}(Left). Note that the vertices of the quadrilaterals $Q_g$ and $Q$ coincide. We denote the vertices as follows: $A_0=\sigma_-=(\gamma_v)_-=\beta_-$ is the identity, $A_1=\sigma_+=(\gamma_u)_-=\alpha_-$, $A_2=\tau_+=(\gamma_u)_+=\alpha_+$, and $A_3=\tau_-=(\gamma_v)_+=\beta_+$. We refer to $\alpha$ and $\beta$ (or $\gamma_u$ and $\gamma_v$) as the \textit{horizontal} edges of $Q$ (or $Q_g$), and to $\sigma$ and $\tau$ as the \textit{vertical} edges.

If $u,v$ and $g$ are written as parabolic words in $S_i$ for some $i$ then each one of the paths $\alpha$, $\beta$, $\gamma_u$, $\gamma_v$, $\sigma$ and $\tau$ consists of a unique parabolic component in $P_i$. In this case, we say that $Q$ and $Q_g$ are \textit{parabolic} (or $P_i$-\textit{parabolic}) quadrilaterals. If $u\in P_i$ and $v\in P_j$ are written as parabolic words but $g\notin P_i\cup P_j$ then we say that $Q$ and $Q_g$ are \textit{semi-parabolic} quadrilaterals. We say that $Q$ and $Q_g$ are \textit{hyperbolic} quadrilaterals if none of $u$ and $v$ is written as a parabolic word.

\begin{Def} \label{def:synchronouscomponents}
Let $\sigma$ and $\tau$ be two distinct paths in the Cayley graph $\Gamma$ of $G$, both labelled by $g$. Let $g=g_1pg_2$, where $p$ is a parabolic component. The corresponding parabolic components $\sigma_p$ of $\sigma$ and $\tau_p$ of $\tau$ are called \textit{synchronous}.
\end{Def}
Throughout the paper, we are interested in shortest conjugating elements. In what follows, we are always looking for a\textit{ shortest conjugating element} (or a \textit{shortest conjugator}) for $u$ and $v$.  What we mean is an element $g$ conjugating a cyclic permutation of $u$ to a cyclic permutation of $v$, such that $g$ is \textit{shortest with respect to the relative length.} Unless stated otherwise, we say that $g$ is a \textit{shortest conjugating element for} $u$ and $v$ if in the quadrilateral $Q$ the relative geodesic $\sigma$ (or $\tau$) is a shortest path connecting a point on $\alpha$ to a point on $\beta$.
%\textbf{Explain why cyclic permutations of $\alpha$ and $\beta$ correspond to cyclic permutations of $u$ and $v$. Why is $\sigma$ a shortest path in both $Q_g$ and $Q$?}

 The following subsets of the elements of $G$ play a role in our computations in sections~\ref{sec:Preliminary} and~\ref{sec:Alg}: 
\begin{itemize}
\item $B(r_1,r_2)=\{w\in F(S)\mid |w|_{\hat{\Gamma}}\leq r_1, |p|_{\Gamma}\leq r_2,\ \text{for each parabolic component}\ p\ \text{of}\ w \}$ is a subset of the ball $B_{r_1r_2}$  of radius $r_1r_2$ in $\Gamma$. This is a proper subset of $B_{r_1r_2}$, whenever $r_1r_2>1$.
\item $B_i=\{p\in P_i\mid |p|_{\Gamma}\leq C(3) \}$ is the set of ``very short'' elements of $P_i$, for each $i=1,2,\dots,m$.
\end{itemize}
\begin{figure}
\begin{subfigure}{0.3\textwidth}
%\label{fig:QgQ}
\setlength{\unitlength}{.4cm}
\begin{picture}(15, 12)
%RECTANGLE
	%LEFT SIDE
		\qbezier(2, 2)(3, 6)(2, 10)					%curve
		\put(2.5, 6){\vector(0, 1){0}}				%arrowhead
		\put(1.5, 6){$\sigma$}						%labeling
	%BOTTOM SIDE
		\qbezier(2, 2)(7, 3)(12, 2)					%curve
		\put(7, 2.5){\vector(1, 0){0}}				%arrowhead
		\put(4.5, 2.75){$\gamma_v$}			%labeling
	%TOP SIDE
		\qbezier(2, 10)(7, 9)(12, 10)				%curve
		\put(8.25, 9.55){\vector(1, 0){0}}				%arrowhead
		\put(9, 9){$\gamma_u$}					%labeling
	%RIGHT SIDE
		\qbezier(12, 2)(11, 6)(12, 10)				%curve
		\put(11.5, 6){\vector(0, 1){0}}				%arrowhead
		\put(11.75, 6){$\tau$}						%labeling
%SQUIGGLY LINES
	%TOP
		\qbezier(2, 10)(4.5, 8)(7, 10)				%left part
		\qbezier(7, 10)(9.5, 12)(12, 10)			%right part
		\put(7, 10){\vector(1, 1){0}}					%arrowhead
		\put(9, 11.25){$\alpha$}					%labeling
	%BOTTOM
		\qbezier(2, 2)(4.5, 0)(7, 2)				%left part
		\qbezier(7, 2)(9.5, 4)(12, 2)				%right part
		\put(7, 2){\vector(1, 1){0}}					%arrowhead
		\put(6.5, 0.5){$\beta$}						%labeling
%VERTEX LABELING
	\put(.75, 1.5){$A_0$}						%bottom left
	\put(.75, 10.0){$A_1$}						%top left
	\put(12.25, 9.75){$A_2$}					%top right
	\put(12.25, 1.75){$A_3$}					%bottom right
\end{picture}
%	\caption{The quadrilaterals \newline $Q=\sigma\circ\alpha\circ\tau^{-1}\circ\beta^{-1}$ and \newline $Q_g=\sigma\circ\gamma_u\circ\gamma_v^{-1}\circ\beta^{-1}$.}
\end{subfigure}
\qquad \qquad \qquad
\begin{subfigure}{0.3\textwidth}
%\label{fig:SelfInter}
\setlength{\unitlength}{.4cm}
\begin{picture}(16, 12)
%lines
	%curved sides - rectangle
	\qbezier(2, 2)(3, 6)(2, 10)
  	\qbezier(2.5, 7)(0.8, 8.5)(2, 10)
%	\qbezier(2, 10)(7, 9)(12, 10)
	\qbezier(12, 2)(11, 6)(12, 10)
	%squiggly lines
		%top
		\qbezier(2.5, 7)(4.5, 8)(7, 10)
		\qbezier(7, 10)(9.5, 12)(12, 10)
		%bottom
		\qbezier(2, 2)(4.5, 0)(7, 2)
		\qbezier(7, 2)(9.5, 4)(12, 2)
	%arrowheads
		%on rectangle sides
		\put(2.5, 6){\vector(0, 1){0}}
		\put(11.5, 6){\vector(0, 1){0}}
		%on squiggly lines
		\put(7, 10){\vector(1, 1){0}}
		\put(7, 2){\vector(1, 1){0}}
		%on x and x_\alpha
		\put(2.4, 8.3){\vector(0, 1){0}}
		\put(1.5, 8.5){\vector(0, -1){0}}
%labeling
	%vertices
	\put(.75, 1.5){$A_0$}
	\put(.75, 10.0){$A_1$}
	\put(12.25, 9.75){$A_2$}
	\put(12.25, 1.75){$A_3$}
	%lines
		%rectangle sides
		\put(1.5, 6){$\sigma$}
		\put(11.75, 6){$\tau$}
		\put(0, 8.5){$x_{\alpha}$}
		\put(2.5, 8.75){$x$}
		%squiggly lines
		\put(9, 11.25){$\alpha$}
		\put(6.5, 0.5){$\beta$}
\end{picture}
%	\caption{This figure illustrates the proof of Lemma.}
\end{subfigure}
\caption{(Left) The quadrilaterals $Q=\sigma\circ\alpha\circ\tau^{-1}\circ\beta^{-1}$ and $Q_g=\sigma\circ\gamma_u\circ\tau^{-1}\circ\gamma_v^{-1}$. (Right) This figure illustrates the proof of Lemma~\ref{lem:noselfintersection}(1).}\label{fig:F1}
\end{figure}
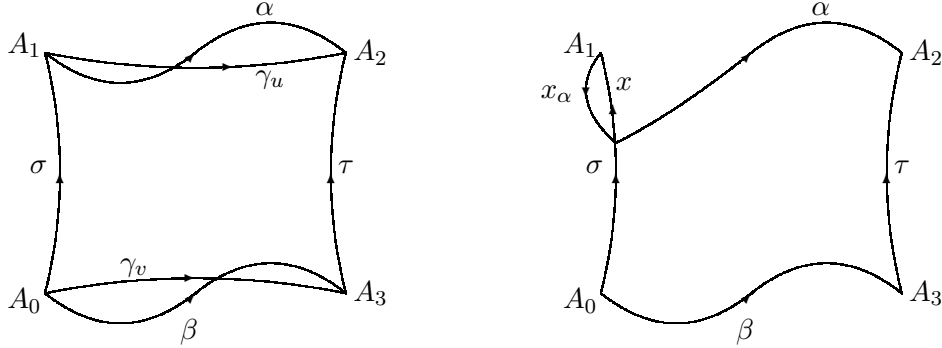

\section{Estimates}\label{HypGpsEstimates}
\subsection{Hyperbolic spaces}
 The following theorem provides useful properties of $k$-local geodesics in a $\delta$-hyperbolic space.
\begin{thm}   \cite[III.H. Theorem1.13]{BridsonHaefliger} \label{thm:klocalgeodesics} Let $X$ be a $\delta$-hyperbolic geodesic space and let $\gamma\colon[a,b]\rightarrow X$ be a $k$-local geodesic, where $k>8\delta$. Then:
\begin{enumerate}
\item $im(\gamma)$ is contained in the $2\delta$-neighbourhood of any geodesic segment $[\gamma(a),\gamma(b)]$ connecting its endpoints.
\item $[\gamma(a),\gamma(b)]$ is contained in the $3\delta$-neighbourhood of $im(\gamma)$, and
\item $\gamma$ is a $(\lambda,\varepsilon)$-quasi-geodesic, where $\varepsilon=2\delta$ and $\lambda=(k+4\delta)/(k-4\delta)$.
\end{enumerate}
\end{thm}

\begin{cor} \label{localgeodesicquadril}  Every side of the quadrilateral $Q$, defined in Section~\ref{sec:notation}, is contained in the $7\delta$-neighbourhood of the other three sides.
\end{cor}
\begin{proof}  In the geodesic quadrilateral $Q_g$ every side is contained in the $2\delta$-neighbourhood of the other three sides. Now we apply the first two assertions of Theorem~\ref{thm:klocalgeodesics} to obtain the claim.
\end{proof}

\begin{lem} \label{lem:noselfintersection} Let $g$ be a shortest conjugating element for (cyclic permutations of) $u$ and $v$. Then:
\begin{enumerate}
\item Each one of the intersections $\alpha\cap\sigma$, $\alpha\cap\tau$, $\beta\cap\sigma$ and $\beta\cap\tau$ is a single point, and this point is a vertex of the quadrilateal $Q$.
\item The paths $\sigma\circ\gamma_u$ and $\tau^{-1}\circ\gamma_v^{-1}$ are $(3,0)$- quasi-geodesics. The paths $\sigma\circ\alpha$ and $\tau^{-1}\circ\beta^{-1}$ are $(2\lambda+1,\varepsilon)$- quasi-geodesics, where $\lambda$ and $\epsilon$ are as in Theorem~\ref{thm:klocalgeodesics}.
\end{enumerate}
\end{lem}
\begin{proof} (1) Suppose $\alpha\cap\sigma$ contains more than just one point. Let $B$ be a point in  $\alpha\cap\sigma$, different from $A_1$. Denote by $x$ the label of the segment of $\sigma$ joining $B$ and $A_1$. Note that whereas the label $x_{\alpha}$ of the subpath of $\alpha$ joining $A_1$ and $B$ may differ from $x$ in $F(S)$, necessarily $x^{-1}=x_{\alpha}$ in $G$. Let $g_1$ and $u_1$ be the labels of the remaining subpaths of $g$ and $u$, correspondingly. We have $1=gug^{-1}v^{-1}=g_1xx_{\alpha}u_1x^{-1}g_1^{-1}v^{-1}=g_1u_1x_{\alpha}g_1^{-1}v^{-1}$. It follows that $g_1$ is a shorter conjugator for a conjugate of $u$ and $v$, which is a contradiction. The other cases are similar.

(2) Let $C$ be an arbitrary point on $\sigma$, and let $D$ be a point on $\alpha$. Note that $d(C,D)\geq d(C,A_1)$ because $g$ is a shortest conjugating element. Also, note that $$l_{\alpha}(A_1,D)\leq \lambda d(A_1,D)+\varepsilon\leq \lambda(d(A_1,C)+d(C,D))+\varepsilon\leq 2\lambda d(C,D)+\varepsilon.$$ It follows from the latter inequality and from the statement (1) of this lemma that $$l_{\sigma\circ\alpha}(C,D) = d(C,A_1)+l_{\alpha}(A_1,D)\leq (2\lambda+1) d(C,D)+\varepsilon.$$ 

To prove the assertion for $\sigma\circ\gamma_u$ and $\tau^{-1}\circ\gamma_v^{-1}$, note that $\gamma_u$ and $\gamma_v$ are $(1,0)$-quasi-geodesics.
\end{proof}

\begin{cor} \label{cor:72delta}  If $\alpha$ and $\beta$ are $(8\delta+1)$-local geodesics then $\alpha$ and $\beta$ are $(3,2\delta)$-quasi-geodesics and the paths $\sigma\circ\alpha$ and $\tau^{-1}\circ\beta^{-1}$ are $(7,2\delta)$- quasi-geodesics.
\end{cor}

The following lemma is the first step to the proof of the existence of a linear bound on the length of a shortest conjugating element, stated in Theorem~\ref{thm:IntroEstimate}. The argument in the proof is due to Lysenok~\cite[Lemma 10]{Lysenok90}.
\begin{lem}\label{lem:linearbound} (Lysenok) Let $u$ and $v$ be conjugate in $G$, and let $g\in G$ be such that $v=gug^{-1}$, and $g$ has the smallest possible relaive length among all the elements that conjugate a cyclic conjugate of $u$ and a cyclic conjugate of $v$. If $|g|_{\hat{\Gamma}}>|u|_{\hat{\Gamma}}+|v|_{\hat{\Gamma}}+4\delta+2$ then both $u$ and $v$ are conjugate in $G$ to an element $z\in G$ with $|z|_{\hat{\Gamma}}\leq 4\delta$.
\end{lem}
\begin{proof}   We consider the quadrilateral $Q$ corresponding to the equality $v=gug^{-1}$, as before. Assume that $|g|_{\hat{\Gamma}}>|u|_{\hat{\Gamma}}+|v|_{\hat{\Gamma}}+4\delta+2$. 
Let $\mu$ be the diagonal in the quadrilateral $Q$ joining the identity $A_0$ and the vertex $A_2$; $\mu$ is labeled by a word $w_{\mu}$ such that $w_{\mu}=gu=vg$, see Figure~\ref{fig:F2}(Left). Let $\sigma_1$ and $\sigma_2$ be such that $\sigma=\sigma_1\circ\sigma_2$, $|\sigma_1|_{\hat{\Gamma}}>|v|_{\hat{\Gamma}}+2\delta$ and  $|\sigma_2|_{\hat{\Gamma}}>|u|_{\hat{\Gamma}}+2\delta$. Then the endpoint $D=t(\sigma_1)$ of $\sigma_1$ is $\delta$-close to a point $M$ on $\mu$, and $M$ is $\delta$-close to a point $T_0$ on $\tau$. Let $\tau_1$ and $\tau_2$ be such that $\tau=\tau_1\circ\tau_2$ and  $|\sigma_1|_{\hat{\Gamma}}=|\tau_1|_{\hat{\Gamma}}$, and let $T_1=t(\tau_1)$. Then necessarily $d_{\hat{\Gamma}}(T_0,T_1)\leq 2\delta$, for if not then either $\sigma_1\circ[D,T_0]\circ[T_0,A_2]$, or $[A_3,T_0]\circ[T_0,D]\circ\sigma_2$ was a path shorter than $\sigma$ joining a point on $\beta$ and a point on $\alpha$, which would be a contradiction. It follows that $d_{\hat{\Gamma}}(D,T_1)\leq 4\delta$. If $z$ is the label of a relative geodesic joining $D$ and $T_1$ then $z$ is conjugate to both $u$ and $v$ and has the required relative length.
\end{proof}
%
% Figure 3 and 4
\begin{figure}
\begin{subfigure}{0.4\textwidth}
%\label{fig:LemmaLysenok}
\setlength{\unitlength}{.4cm}
\begin{picture}(16, 22)
%RECTANGLE
	%LEFT SIDE
		\qbezier(2, 2)(3, 6)(2, 20)					%curve
		\put(2.5, 6){\vector(0, 1){0}}				%arrowhead
		\put(1.5, 6){$\sigma_1$}						%labeling
		\put(2.25, 16){\vector(0, 1){0}}				%arrowhead
		\put(1.1, 16){$\sigma_2$}						%labeling
		\put(2.4,12.9){\circle*{.2}}	                          % bold vertex
		\put(1.25, 13.15){$D$}	                         % labeling of bold vertex
	%RIGHT SIDE
		\qbezier(12, 2)(11, 6)(12, 20)				%curve
		\put(11.5, 6){\vector(0, 1){0}}				%arrowhead
		\put(11.75, 6){$\tau_1$}						%labeling
		\put(11.75, 16){\vector(0, 1){0}}				%arrowhead
		\put(12.25, 16){$\tau_2$}						%labeling
		\put(11.5,9.7){\circle*{.2}}	                          % bold vertex
		\put(11.85, 9.25){$T_0$}	                         % labeling of bold vertex
		\put(11.55,12.9){\circle*{.2}}	                          % bold vertex
		\put(11.95, 13.15){$T_1$}	                         % labeling of bold vertex
   %DIAGONAL
       	\qbezier(2, 2)(7, 11)(12, 20)					%curve
     	\put(9.6, 15.5){\vector(1, 1){0}}				%arrowhead
        \put(8.2, 16){$\mu$}						%labeling
        \put(7.0,11){\circle*{.2}}	                          % bold vertex
        \put(6, 11.45){$M$}	                         % labeling of bold vertex
     %LITTLE SEGMENTS ACROSS
     	\qbezier(2.4, 12.9)(4, 11)(7.0,11)					%curve from D to M
     	\qbezier(7.0,11)(9.3, 10.2)(11.5,9.7)		%curve from M to T_0	
%SQUIGGLY LINES
	%TOP
		\qbezier(2, 20)(4.5, 18)(7, 20)				%left part
		\qbezier(7, 20)(9.5, 22)(12, 20)			%right part
		\put(7, 20){\vector(1, 1){0}}					%arrowhead
		\put(9, 21.25){$\alpha$}					%labeling
	%BOTTOM
		\qbezier(2, 2)(4.5, 0)(7, 2)				%left part
		\qbezier(7, 2)(9.5, 4)(12, 2)				%right part
		\put(7, 2){\vector(1, 1){0}}					%arrowhead
		\put(6.5, 0.5){$\beta$}						%labeling
%VERTEX LABELING
	\put(.75, 1.5){$A_0$}						%bottom left
	\put(.75, 20.0){$A_1$}						%top left
	\put(12.25, 19.75){$A_2$}					%top right
	\put(12.25, 1.75){$A_3$}					%bottom right
\end{picture}
%	\caption{If $d(T_0,T_1)>2\delta$ then $[A_3,T_0]\circ[T_0,D]\circ\sigma_2$ is shorter than $\sigma$, see Lemma.}
\end{subfigure}
\qquad \qquad 
\begin{subfigure}{0.4\textwidth}
%\label{fig:Lemma312}
\setlength{\unitlength}{.4cm}
\begin{picture}(16, 22)
% TOP RECTANGLE
		\put(6.75, 17.3){$Q_i$}						%labeling
	%LEFT SIDE
		\qbezier(2, 14)(3, 16)(2, 20)					%curve
		\put(2.5, 17){\vector(0, 1){0}}				%arrowhead
		\put(1.5, 16){$f_i$}						%labeling
	%BOTTOM SIDE
		\qbezier(2, 14)(7, 13)(12, 14)					%curve
		\put(4, 13.7){\vector(1, 0){0}}				%arrowhead
		\put(3.5, 12.75){$a$}			%labeling
		\put(10, 13.7){\vector(1, 0){0}}				%arrowhead
		\put(9.5, 12.75){$b$}			%labeling
		\put(5,13.6){\circle*{.3}}	                          % bold vertex
	%TOP SIDE
			\qbezier(2, 20)(7, 19)(12, 20)					%curve
				\put(7, 19.5){\vector(1, 0){0}}				%arrowhead
				\put(6.5, 20.75){$c_i$}			%labeling
	%RIGHT SIDE
		\qbezier(12, 14)(11, 16)(12, 20)				%curve
		\put(11.6, 17.5){\vector(0, 1){0}}				%arrowhead
		\put(11.75, 16){$f_i$}						%labeling
% BOTTOM RECTANGLE
	\put(6.75, 6.3){$Q_{i+1}$}						%labeling
    %LEFT SIDE
		\qbezier(2, 2)(3, 5)(2, 8)					%curve
		\put(2.5, 5){\vector(0, 1){0}}				%arrowhead
		\put(.5, 5){$f_{i+1}$}						%labeling
	%RIGHT SIDE
		\qbezier(12, 2)(11, 7)(12, 8)				%curve
		\put(11.6, 5.5){\vector(0, 1){0}}				%arrowhead
		\put(11.75, 6){$f_{i+1}$}						%labeling
	%BOTTOM SIDE
		\qbezier(2, 2)(7, 3)(12, 2)					%curve
			\put(7, 2.5){\vector(1, 0){0}}				%arrowhead
			\put(6.5, .75){$c_{i+1}$}			%labeling
	% TOP SIDE
		\qbezier(2, 8)(7, 9)(12, 8)					%curve
		\put(6, 8.5){\vector(1, 0){0}}				%arrowhead
		\put(5, 8.75){$b$}			%labeling
		\put(10.5, 8.3){\vector(1, 0){0}}				%arrowhead
		\put(10, 8.75){$a$}			%labeling
		\put(9,8.45){\circle*{.3}}	                          % bold vertex
\end{picture}
%	\caption{Need to adjust $Q_i$ before gluing, see Lemma.}
\end{subfigure} 
\caption{(Left) If $d(T_0,T_1)>2\delta$ then $[A_3,T_0]\circ[T_0,D]\circ\sigma_2$ is shorter than $\sigma$, see Lemma~\ref{lem:linearbound}. (Right) Need to adjust $Q_i$ before gluing, see Lemma~\ref{lem:conjugatorforhyperbolics}.}\label{fig:F2}
\end{figure}
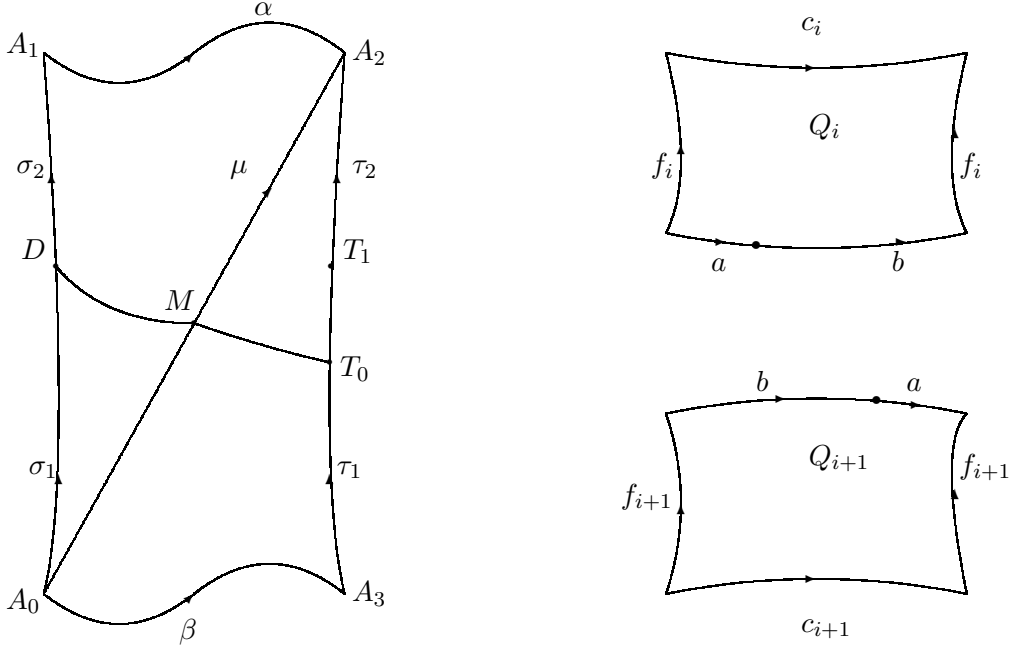

The following lemma is stated in \cite[III.$\Gamma.$ Lemma 2.11]{BridsonHaefliger} for hyperbolic groups; the proof below is analogous to the proof of \cite[III.$\Gamma.$ Lemma 2.9]{BridsonHaefliger}, and we include it for completeness.

\begin{lem}\label{lem:hyperbolicdichotomy} Let $\Gamma$ be a $\delta$-hyperbolic metric space such that paths in $\Gamma$ are labeled by elements of a group $G$. Let $\alpha$ and $\beta$ be  cyclic  $(8\delta+1)$-local geodesic paths in $\Gamma$ with the labels $lab(\alpha)=u$ and $lab(\beta)=v$. Suppose that $u$ and $v$ are conjugate in $G$, and let $\sigma$ and $\tau$ be two geodesics
in $\Gamma$ with $lab(\sigma)=lab(\tau)=g$ such that $gug^{-1}v^{-1}=1$ in $G$, $\sigma\circ\alpha\circ\tau^{-1}\circ\beta^{-1}$ is a closed path in $\Gamma$, and for any geodesic $\rho\in \Gamma$ joining a point on $\alpha$ and a point on $\beta$, $l_{\rho}\geq l_{\sigma}$. Then 
\begin{enumerate}
\item $\max\{l_{\alpha},l_{\beta}\}\leq 86\delta+3$, or else
\item $l_{\sigma}\leq 7\delta$.
\end{enumerate}
If $\Gamma$ is a Cayley graph of $G$ then there exists a word $g$ of length at most $7\delta+1$ such that $gu'g^{-1}=v'$ in $G$, where $u'$ and $v'$ are cyclic permutations of $u$ and $v$.
\end{lem}
\begin{proof} W.l.o.g., assume that $l_{\alpha}\geq l_{\beta}$. Let $B$ be the midpoint of $\alpha$.  By Corollary~\ref{localgeodesicquadril}, $B$ is $7\delta$-close to a point $E$ on one of the other three sides of $Q$. Suppose that $E\in\sigma$, and let $A_1=\sigma_+=\alpha_-$ be the vertex of $Q$ where $\sigma$ and $\alpha$ meet; this vertex is unique, by Lemma~\ref{lem:noselfintersection}(1). Note that $d(A_1,E)\leq d(B,E)\leq 7\delta$ because $\sigma$ is a shortest path joining a point on $\alpha$ and a point on $\beta$, by assumption. It follows that $d(A_1,B)\leq 14\delta$, which implies that $d(A_1,A_2)\leq 28\delta+1$, where $A_2=\alpha_+$. Since by Corollary~\ref{cor:72delta}, $\alpha$ is a $(3,2\delta)$-quasi-geodesic, we have that 
$l_{\alpha}\leq 3 d(A_1,A_2)+2\delta\leq 86\delta+3.$
 On the other hand, if $l_{\alpha}\geq 86\delta+3$ then necessarily $d(A_1,A_2)>28\delta+1$, and the argument above shows that $B$ is $7\delta$-close to a point on $\beta$, which proves (2). 
To prove the last assertion note that the vertices of $\Gamma$ closest to the points $B$ and $E$, correspondingly, are at most distance $7\delta+1$ apart. Also, the length of a path joining two vertices in the Cayley graph equals the length of its label in the word metric.
\end{proof}

\subsection{Relatively hyperbolic groups}\label{sec:RelHypGps}

\begin{lem} \label{lem:localgeodnobacktracking}
If $\alpha$ is a relative $(8\delta+1)$-local geodesic then $\alpha$ does not backtrack. In other words, if $\alpha$ travels a nonzero distance inside a left coset $gP$ of a parabolic subgroup $P$ then $\alpha$ never returns to $gP$ after leaving it.
\end{lem}
\begin{proof} By Corollary~\ref{cor:72delta}, $\alpha$ is a relative $(3,2\delta)$-quasi-geodesic. Suppose that $\alpha$ backtracks, so that $lab(\alpha)=lab(\alpha_1)lab(\gamma_1) lab(\alpha_2) lab(\gamma_2)lab(\alpha_3)$, where $lab(\gamma_1),lab(\gamma_2)\in P$, for some parabolic subgroup $P\in\{P_1,P_2,\dots,P_m \}$, $lab(\gamma_1)\neq 1$, $\alpha_2$ is a nonempty path with $lab(\alpha_2)\notin P$, and $lab(\alpha_1)P=lab(\alpha_1) lab(\gamma_1) lab(\alpha_2)P$. It follows that $lab(\alpha_2)=lab(\gamma_0)\in P$; in particular, the relative distance between the endpoints of $\alpha_2$ does not exceed 1. Since $\alpha_2$ is a relative $(3,2\delta)$-quasi-geodesic, $\ell_{\hat{\Gamma}}(\alpha_2)\leq 3+2\delta<8\delta+1$. It follows that $\alpha_2$ is a relative geodesic, so that $\ell_{\hat{\Gamma}}(\alpha_2)=1$. We conclude that $2\leq \ell_{\hat{\Gamma}}(\gamma_1\circ\alpha_2\circ\gamma_2)\leq 3$, depending on whether or not $\gamma_2$ is the empty path. Therefore, $\gamma_1\circ\alpha_2\circ\gamma_2$ is a relative geodesic of length 2 or more, whereas $lab(\gamma_1\circ\alpha_2\circ\gamma_2)\in P$, which is a contradiction.
\end{proof}
In the following lemma we explore the case when $fuf^{-1}=v$ for some $f\in G$, and in the quadrilateral $Q$ (some) parabolic components of $\alpha$ and $\beta$ are connected. 
\begin{lem} \label{lem:veryshortg}
Let $\alpha$ and $\beta$ be cyclic relative $(8\delta+1)$-local geodesics representing $u$ and $v$, respectively. Let $u$ and $v$ be conjugate in $G$, so that $v=fuf^{-1}$ for some $f\in G$. Suppose that $\alpha=\alpha_1\circ\xi_u\circ\alpha_2$ $\beta=\beta_1\circ\xi_v\circ\beta_2$, where $p_u=lab(\xi_u)$ and $p_v=lab(\xi_v)$ are in $P_i$, both written as words in $F(S_i)$, for some $i$, while $u_j=lab(\alpha_j)$ and $v_j=lab(\beta_j)$, $j=1,2$, are not in $P_i$. 

If in the quadrilateral $Q_f$ the parabolic components $\xi_u$ and $\xi_v$ are connected then there is a conjugating element $g\in P_i$, written as a word in $F(S_i)$. Moreover, if $g$ has the shortest possible relative length then exactly one of the following holds: 
\begin{enumerate}
\item $g=1$ and $u=v$;
\item Both $u$ and $v$ are written as hyperbolic elements, so that $u_j\neq 1$ and $v_k\neq 1$ for some $j,k\in\{1,2\}$, $u\neq v$, and $1\leq |g|_{\Gamma}\leq C(7,2\delta)$;
\item Both $u$ and $v$ are written as parabolic elements, so that $u=p_u$ and $v=p_v$, and $Q$ is a parabolic quadrilateral.
\end{enumerate}
\end{lem}
\begin{proof} Suppose that $u\neq v$, so that $g\neq 1$. 

Firstly, we show that if $v=p_v$ then necessarily $u=p_u$. By way of contradiction, suppose that $u=u_1p_uu_2$, where $p_u$ is a maximal parabolic component, and at least one of $u_1$ and $u_2$ is not equal 1 and is not in $P_i$. Since $p_u$ and $p_v$ are connected, there is $g\in P_i$ such that $g=fu_2^{-1}$ and therefore, $p_v=fu_1p_uu_2f^{-1}=gu_2u_1p_ug^{-1}$. Whereas $u_2$ and $u_1$ are not written as parabolic words in $F(S_i)$, note that $u_2u_1=g^{-1}p_vgp_u^{-1}\in P_i$. In particular, the relative $(8\delta+1)$-local geodesic path $(\alpha_2\circ\alpha_1\circ\gamma_u)^{-1}$ backtracks, which contradicts Lemma~\ref{lem:localgeodnobacktracking}. Hence, $u_1=u_2=1$, and we have (3).

Now, let $g=v_2fu_2^{-1}$, so that $g\in P_i$ and $gu_2u_1p_ug^{-1}=v_2v_1p_v$, and let $Q$ be the corresponding quadrilateral. Assume that $Q$ is not a parabolic quadrilateral. By the preceding paragraph, this implies that both $u$ and $v$ are written as hyperbolic elements. We have that $v_2v_1p_vgp_u^{-1}=u_2u_1$, and let $q=p_vgp_u^{-1}$. Denote by $\sigma_g$ (or $\tau_q$) the parabolic geodesic labelled by $g$ (or $q$). By Lemma~\ref{lem:localgeodnobacktracking}, $v_2v_1P_i\neq P_i$, hence in the closed path $\sigma_g\circ\alpha_2\circ\alpha_1\circ\tau_q^{-1}\circ\beta_1^{-1}\circ\beta_2^{-1}$ the parabolic components $\sigma_g$ and $\tau_q$ are not connected. Since, according to Corollary~\ref{cor:72delta}, $\sigma_g\circ\alpha_2\circ\alpha_1$ and $\beta_2\circ\beta_1\circ\tau_q$ is a pair of relative $(7,2\delta)$-quasi-geodesics with common endpoints, (2) follows.
\end{proof}
\begin{cor} \label{cor:veryshortg}
Let conjugate elements $u$ and $v$ be represented by relative geodesics $\gamma_u$ and $\gamma_v$, respectively, and let the labels $lab(\gamma_u)$ and $lab(\gamma_v)$ be written as hyperbolic elements.  Suppose that parabolic components of $\gamma_u$ and $\gamma_v$ are connected, and let $g$ be a shortest possible conjugating parabolic element as in the statement of Lemma~\ref{lem:veryshortg}. If $u\neq v$ then $1\leq |g|_{\Gamma}\leq C(3)$.
\end{cor}
\begin{proof} The claim follows immediately from the proof of Lemma~\ref{lem:veryshortg}, one only needs to note that, by Corollary~\ref{cor:72delta}, the geodesic quadrilateral $Q_g$ is formed by two relative $(3,0)$-quasi-geodesics with common endpoints.
\end{proof}

\begin{lem}\label{lem:parabcompsofg} Let $g$ be a conjugating element of shortest relative length for cyclic permutations of $u$ and $v$. Suppose that $Q$ (or $Q_g$) is not a parabolic quadrilateral. 
\begin{enumerate}
\item\label{lem:parabcompsofgPart1} If $u\notin F(S_i)$, for all $i$, so that $\alpha$ does not consist of a single parabolic component then either $\alpha$ and $\sigma$, or $\alpha$ and $\tau$ have no connected parabolic components.
\item \label{lem:parabcompsofgPart2}  If $u\notin F(S_i)$ and $v\notin F(S_j)$, for all $i,j$, and $\sigma$ and $\tau$ do not have connected parabolic components, then every parabolic component of $g$ appears isolated in either $\sigma$ or $\tau$.
\end{enumerate}
\end{lem}
\begin{proof} (1) Suppose that $\sigma$ and $\alpha$ have connected parabolic components, $p_s\in P_i$ and $p_a\in P_i$, correspondingly. Abusing notation, we denote by $p_s$ and $p_a$ both paths in $\Gamma$ and their labels. Let $\sigma=\sigma_1\circ p_s\circ\sigma_2$ and $\alpha=\alpha_1\circ p_a\circ\alpha_2$. Since by Lemma~\ref{lem:localgeodnobacktracking}, $(8\delta+1)$-local geodesics do not bactrack, $\sigma_2$ and $\alpha_1$ are either both empty paths, or both non-empty paths. Indeed, if $\sigma_2$ was the empty path and $\alpha_1$ was non-empty then $\sigma_1P_i=\sigma_1p_s\alpha_1P_i$ would imply that $\alpha$ bactracks, and this is a contradiction. The other case is similar. On the other hand, if both $\sigma_2$ and $\alpha_1$ were non-empty then $\sigma_1\circ[(p_s)_-,(p_a)_+]$ would be a shorter conjugating element for cyclic conjugates of $u$ and $v$, which is a contradiction.  So, suppose that both $\sigma_2$ and $\alpha_1$ are empty paths. 

To show that $\alpha$ and $\tau$ have no connected parabolic components, suppose by way of contradiction that $\alpha=\alpha'_1\circ p_b\circ\alpha'_2$ and $\tau=\tau_1\circ p_t\circ\tau_2$ and that the parabolic components $p_b\in P_i$ and $p_t\in P_i$ are connected. The arguments above apply to show that $\tau_2^{-1}$ and $\alpha'_2$ are both empty paths. Note that we only need to consider the case when $p_t$ and $p_s$ are synchronous components, hence they have the same label. It follows that $p_a$ and $p_b$ belong to the same parabolic subgroup. If the parabolic components $p_b$ and $p_a$ are distinct then $\alpha$ is not a relative cyclic $(8\delta+1)$-local geodesic, as $p_a^{-1}up_a$ would have a shorter relative length, which is a contradiction. However, if $p_a$ and $p_b$ are one and the same parabolic component of $\alpha$ then $\alpha=p_a=p_b$, which contradicts our assumption on $u$.

(2) The same arguments apply to show that if $v\notin P_j$, for all $j$, and $\sigma$ and $\beta$ have connected parabolic components $p_t$ and $p_v$, correspondingly, then necessarily the component of $\tau$ synchronous with $p_t$ is isolated. The claim follows.
\end{proof}

\begin{lem}\label{lem:isolated} Let $g$ be a conjugating element of shortest relative length for cyclic permutations of $u$ and $v$. Suppose that $u\notin F(S_i)$ for all $i=1,2,\dots,m$. Then every parabolic component of $g$ appears isolated in either $\sigma$ or $\tau$, in each one of the following cases.
\begin{enumerate}
\item \label{lem:isolatedhyperbolicelts} The  elements $u$ and $v$ are hyperbolic,
\item \label{lem:isolatedparabolicelt}  The  element $v=p$ consists of a single parabolic component, $u$ is written as a hyperbolic element, and $g$ has a shortest relative length among all the elements of $G$ that conjugate $u$ into a parabolic subgroup of $G$. Moreover, in this case $p$ is an isolated component of the path $\tau^{-1}\circ\beta^{-1}$.
\end{enumerate}
\end{lem}
\begin{proof}  (1) Note that connected parabolic components of $\sigma$ and $\tau$ can only be synchronous, by~\cite[Lemma 3.39]{OsinRelHypGps}. However, if $\sigma$ and $\tau$ have synchronous connected parabolic components then $u$ and $v$ are parabolic, which is a contradiction. The claim now follows from Lemma~\ref{lem:parabcompsofg}(2).

(2)  Parabolic components of $\sigma$ and $\tau$ are not connected, for if they were then $g$ would not be a shortest conjugator. Moreover, parabolic components of $\sigma$ and $\tau$ are not connected to $p$ for if they were connected then, by an argument in the proof of Lemma~\ref{lem:parabcompsofg}, necessarily $u=g_1^{-1}p_1^{-1}pp_1g_1$ for some $p_1\in P_i$. However, this implies that $u=g_1^{-1}(p_1^{-1}pp_1)g_1$ is conjugate into $P_i$ by the shorter element $g_1$, which is a contradiction. Now, Lemma~\ref{lem:parabcompsofg} applies to prove the claim.
\end{proof}

Recall that for $i=1,2,\dots,m$, we denote by $B_i$ the following subset of the parabolic subgroup: $B_i=\{p\in P_i\mid |p|_{\Gamma}\leq C(3) \}$.
\begin{lem}\label{lem:conjugateparabolic} Let $u\in F(S_i)$ and $v\in F(S_j)$ be conjugate in $G$. If $i=j$ then we assume that $u$ and $v$ are not conjugate in $P_i$. There are $p_u\in[u]_{P_i}\cap B_i$ and $p_v\in[v]_{P_j}\cap B_j$.
\end{lem}
\begin{proof} The relative length of both $u$ and $v$ equals 1, so that in this case, the semi-parabolic quadrilateral $Q$ is the concatenation of two $(2,0)$-quasi-geodesics. In general, $v=fuf^{-1}$, where $f=f_vgf_u$ with $f_u\in P_i$ and $f_v\in P_j$. We denote by $p_u=f_uuf_u^{-1}$ and by $p_v=f_v^{-1}vf_v$. We assume that $g$ is a shortest conjugating element  for $p_u$ and $p_v$. By assumption, $g\notin P_i$. If $v\in P_j$ with $j\neq i$ then $g\notin P_j$ because $p_u=g^{-1}p_vg\notin P_j$. Hence, if the relative length of $g$ equals 1 then, by Lemma~\ref{lem:veryshortg}, all the parabolic components of $Q$ are isolated, and we conclude that $|p_u|_{\Gamma}\leq C(2)$ and $|p_v|_{\Gamma}\leq C(2)$; note that $C(2)\leq C(3)$.

So, we can assume that $|g|_{\hat{\Gamma}}>1$. If a parabolic component of $g$ is connected to $u$ then, by an argument in the proof of Lemma~\ref{lem:parabcompsofg}, necessarily $v=g_1p_1up_1^{-1}g_1^{-1}$ for some $p_1\in P_i$. Set $p_u=p_1up_1^{-1}$ to have $v=g_1p_ug_1^{-1}$, where no parabolic component of $g_1$ connected to $p_u$. If no parabolic component of $g$ is connected to $u$ then we set $p_u=u$ and $g_1=g$.  Similarly, if a parabolic component of $g_1$ is connected to $v\in P_j$ then there is $p_2\in P_j$ and $p_v=p_2^{-1}vp_2$ such that $p_u=g_2^{-1}p_vg_2$, where $g_2=p_2^{-1}g_1$ and no parabolic component of $g_2$ is connected to $p_v$. We set $p_v=v$ and $g_2=g_1$ if parabolic components of $g_1$ are not connected to $v$. 

Let $\sigma'$ and $\tau'$ represent $g_2$, and let $\hat\gamma_u$ and $\hat\gamma_v$ represent $p_u$ and $p_v$, correspondingly.
Suppose that parabolic components of $\sigma'$ and $\tau'$ are connected. By ~\cite[Lemma 3.39]{OsinRelHypGps}, the connected components have to be synchronous. Let $p,q\in P_1\cup P_2\cup\dots\cup P_m$ be such that $p_u=h^{-1}ph$, $p_v=fqf^{-1}$, and let the paths $\sigma_h$ and $\tau_h$, as well as $\sigma_f$ and $\tau_f$, have no connected components. Then in the quadrilaterals $Q_1=\sigma_h\circ\hat\gamma_u\circ\tau_h^{-1}\circ\gamma_p$ and $Q_2=\sigma_f\circ\gamma_q\circ\tau_f^{-1}\circ\hat\gamma_v^{-1}$ the parabolic components are all isolated, by an argument in the proof of Lemma~\ref{lem:isolated}(\ref{lem:isolatedparabolicelt}). The claim for $|p_u|_{\Gamma}$ and $|p_v|_{\Gamma}$ follows since each of $Q_1$ and $Q_2$ is the concatenation of two $(2,0)$-quasi-geodesics. 
\end{proof}

\begin{cor} \label{cor:conjugatorforparabolic}
Let $u$ and $v$ be as in the statement of Lemma~\ref{lem:conjugateparabolic}, and let $g$ be a shortest conjugator for $u$ and $v$, so that $v=gug^{-1}$. Assume that the corresponding semi-parabolic  quadrilateral $Q$ is minimal, that is, the parabolic components of $\sigma$ and $\tau$ are not connected. Then the $\Gamma$-length of every parabolic component of $g$ is bounded by $C(3)$.
\end{cor}
\begin{proof} The claim follows from the proof to Lemma~\ref{lem:conjugateparabolic}.
\end{proof}

\subsection{Long elements}
The following theorem provides an estimate on the length of a shortest conjugating element $g$ for cyclic conjugates of $u$ and $v$ in $G$, if $u$ and $v$ are long.
\begin{thm}\label{thm:conjugatorforlongelts} Let $G$ be a group hyperbolic relative to a set of subgroups $\{P_1,\dots,P_m\}$, and let $u$ and $v$ be two elements conjugate in $G$. Let $g\in G$ be a shortest conjugator for cyclic conjugates $u'$ and $v'$ of $u$ and $v$. Let $\alpha$ and $\beta$ be relative cyclic $(8\delta+1)$-local geodesics representing $u'$ and $v'$, correspondingly. If $\max\{\ell_{\alpha},\ell_{\beta}\}>86\delta+3$ then $u$ and $v$ are hyperbolic elements of $G$, and $g\in B(7\delta+1,C(7,2\delta))$. In particular, the length of $g$ can be bounded above as follows:
$$ |g|_{\Gamma}\leq (7\delta+1)C(7,2\delta).$$
\end{thm}
\begin{proof} Without loss of generality, suppose that $\ell_{\alpha}\geq \ell_{\beta}$.  Since, by Corollary~\ref{cor:72delta}, $\alpha$ and $\beta$ are relative cyclic $(3,2\delta)$-quasi-geodesics, the relative distance between the endpoints  $A_1$ and $A_2$ of $\alpha$ can be estimated in terms of the length $l_{\alpha}$ as follows: 
$d_{\hat{\Gamma}}(A_1,A_2)\geq \dfrac{l_{\alpha}-2\delta}{3}.$

In particular, since $l_{\alpha}\geq 86\delta+3$, the relative distance between the endpoints of $\alpha$ is at least $28\delta+1$. If $u$ and $v$ are conjugate then by Lemma~\ref{lem:hyperbolicdichotomy},  there is a conjugator of the relative length $7\delta+1$ or less, and it follows that the relative distance between the endpoints of $\beta$ is at least $28\delta+1-2(7\delta+1)>14\delta>1$. It follows that $u$ and $v$ are not parabolic elements.

The upper bound on the $\Gamma$-length of $g$ follows from Lemma~\ref{lem:hyperbolicdichotomy}, Corollary \ref{cor:72delta} and Lemma~\ref{lem:isolated}(\ref{lem:isolatedhyperbolicelts}).
\end{proof}
\subsection{Short elements} In this section we consider the case when $u$ and $v$ are such that $\max\{l_{\alpha},l_{\beta}\}\leq 86\delta+3$.  We assume that $u$ and $v$ are represented by relative geodesics $\gamma_u$ and $\gamma_v$, and consider the geodesic quadrilateral $Q_g$. In this section, by a shortest conjugating element for $u$ and $v$ we mean the label of a shortest relative geodesic $\sigma$ joining a point on $\gamma_u$ and a point on $\gamma_v$. If $Q_g$ is not a parabolic quadrilateral and  $\max\{|\gamma_{u}|_{\hat{\Gamma}},|\gamma_{v}|_{\hat{\Gamma}}\}\leq 4\delta$ then we call $Q_g$ a \textit{short base quadrilateral}. In particular, every semi-parabolic quadrilateral is short base. We say that a short base quadrilateral is \textit{minimal} if it does not contain any smaller short base quadrilateral. Note that in a minimal quadrilateral the synchronous parabolic components of the vertical sides are not connected.
\begin{lem}\label{lem:conjugatorforhyperbolics} Let $u$ and $v$ be hyperbolic elements such that $\max\{|\gamma_{u}|_{\hat{\Gamma}},|\gamma_{v}|_{\hat{\Gamma}}\}\leq 4\delta$, and let $f$ be a shortest conjugating element for $u$ and $v$. Then $f\in B(K^{(hyp)}_{4\delta}, C(3))$, where the bound on the relative length of $f$, $K^{(hyp)}_{4\delta}=\# B(4\delta, 2C(3))(16\delta+2)$, does not depend on $u$ and $v$. In particular, the length of $f$ can be bounded as follows:
\[
|f|_{\Gamma}\leq K^{(hyp)}_{4\delta}C(3)=\# B(4\delta, 2C(3))(16\delta+2)C(3).
\]
\end{lem}
\begin{proof} The quadrilateral $Q_f$, corresponding to the equality $u=fvf^{-1}$, consists of a sequence $Q_1,Q_2,\dots Q_s$ of minimal short base quadrilaterals glued along their horizontal sides. We number the minimal quadrilaterals so that the top horizontal side of $Q_1$ is $\gamma_u$ and the bottom horizontal side of $Q_s$ is $\gamma_v$. By Lemma~\ref{lem:linearbound}, the relative length of the vertical side $\gamma_i$ of $Q_i$ is bounded as follows: $|\gamma_{i}|_{\hat{\Gamma}}\leq 12\delta+2$. To bound the length of parabolic components, firstly consider each $Q_i$ separately. We replace the horizontal sides of $Q_i$ by cyclic geodesics and replace each vertical side of $Q_i$ by a shortest conjugator for cyclic permutations of the horizontal sides. When doing this, we may need to replace the horizontal sides of $Q_i$ by their cyclic conjugates. Now each $Q_i$ satisfies the conditions of Lemma~\ref{lem:noselfintersection} and of Lemma~\ref{lem:isolated}(1), and we conclude that the length of each parabolic component of the vertical side of $Q_i$ is bounded by $C(3),\ \forall i$. So, if $c_i$ is a horizontal side of $Q_i$ then $c_i\in B(4\delta,2C(3))$, for all $i$. However, the bottom side of $Q_i$ may not be the same cyclic conjugate as the top of $Q_{i+1}$, see Figure~\ref{fig:F2}(Right). When we glue the adjusted minimal rectangles, we may need to cyclically shift some of their horizontal sides and thus enlarge the relative length of the vertical sides. Let $Q_i$ correspond to the equality $ab=f_ic_if_i^{-1}$ and $Q_{i+1}$ correspond to the equality $c_{i+1}=f_{i+1}baf_{i+1}^{-1}$. We replace $f_i$ by $f'_i=a^{-1}f_i$ if $|a|_{\Gamma}\leq |b|_{\Gamma}$; otherwise, we replace $f_{i+1}$ by $f'_{i+1}=f_{i+1}b$. The length of the parabolic components does not change. The relative length of each vertical side can grow by $4\delta$ at most. It follows that $f'_i\in  B(16\delta+2,C(3))$.

The horizontal sides $c_i$ of the minimal rectangles are all distinct because $f$ is shortest possible. Therefore, $|f|_{\hat{\Gamma}}\leqslant\# B(4\delta, 2C(3))(16\delta+2)$, and the claim follows.
\end{proof}
A careful analysis of the argument in the proof of Lemma~\ref{lem:conjugatorforhyperbolics} shows that the requirement in the statement of the lemma that $u$ and $v$ be hyperbolic elements can be dropped, so long as no horizontal side of any of the quadrilaterals $Q_i$ consists of a single parabolic component. More precisely, we have the following.
\begin{cor}\label{cor:shortbasetower}
Let a quadrilateral $Q_f$ consist of a sequence $Q_1,Q_2,\dots Q_s$ of minimal short base quadrilaterals such that in each $Q_i$, the labels of the horizontal sides are written as hyperbolic words and the vertical sides have identical labels. Let $Q_1,Q_2,\dots Q_s$ be glued along their horizontal sides, as described in the proof of Lemma~\ref{lem:conjugatorforhyperbolics}. Assume that the horizontal sides of the minimal quadrilaterals are all distinct. Then the length of the vertical side $f$ of $Q_f$ is bounded by the constant from the statement of Lemma~\ref{lem:conjugatorforhyperbolics}.
\end{cor}
\begin{Def}\label{def:Ki}
For each $i=1,2,\dots,m$, we define a constant $K_i$ as follows: if $t_1,t_2\in P_i$  are conjugate in $P_i$ and $|t_1|_{\Gamma},|t_2|_{\Gamma}\leq C(3)$ then there is $t\in P_i$ such that $t_1=tt_2t^{-1}$ and $|t|_{\Gamma}\leq K_i$.
\end{Def}
\begin{lem}\label{lem:conjugatorfor4delta}  Let $u$ and $v$ be parabolic elements such that $\max\{|\gamma_{u}|_{\hat{\Gamma}},|\gamma_{v}|_{\hat{\Gamma}}\}\leq 4\delta$. Assume that if $u\in\cup_{i=1}^m F(S_i)$ then $|\gamma_{u}|_{\Gamma}\leq C(3)$ and similarly, if $v\in \cup_{i=1}^m F(S_i)$ then $|\gamma_{v}|_{\Gamma}\leq C(3)$. 
If $h$ is a shortest conjugating element for $u$ and $v$ then $h\in B(K_{4\delta}, K)$, where the bound on the relative length $K_{4\delta}=K^{(hyp)}_{4\delta}+\sum\limits_{i=1}^{m} |S_i|^{C(3)}$ does not depend on $u$ and $v$, $K^{(hyp)}_{4\delta}$ is as in Lemma~\ref{lem:conjugatorforhyperbolics} and $K=\max\{C(3),K_1,K_2,\dots, K_m\}$ for $K_1,K_2,\dots, K_m$ as in Definition~\ref{def:Ki}. More precisely, the length of $h$ can be bounded as follows:
\[
|h|_{\Gamma}\leq \# B(4\delta, 2C(3))(16\delta+2)C(3)+\frac 12\sum\limits_{i=1}^{m}K_i\cdot|S_i|^{C(3)}.
\]
\end{lem}
\begin{proof} With the notation from the proof of Lemma~\ref{lem:conjugatorforhyperbolics}, the minimal quadrilaterals $Q_i$ can have parabolic sides glued to parabolic quadrilaterals. In particular, some $Q_i$ may be semi-parabolic. By Lemma~\ref{lem:isolated}(2), Corollary~\ref{cor:conjugatorforparabolic} and Corollary~\ref{cor:shortbasetower}, the estimates provided in Lemma~\ref{lem:conjugatorforhyperbolics} apply to the minimal quadrilaterals $Q_i$. Since by Lemma~\ref{lem:conjugateparabolic} and by the assumption, the parabolic horizontal sides of all  quadrilaterals are always bounded by $C(3)$, the $\Gamma$-length of a vertical side in each parabolic quadrilateral is bounded by $K_i$ for the corresponding $i$. Note that since $h$ is a shortest conjugator, the intermediate parabolic horizontal sides of the quadrilaterals are all distinct. Hence, the number of parabolic $P_i$-quadrilaterals that may occur is bounded by $\frac 12|S_i|^{C(3)}$. The claim follows.
\end{proof}

\begin{thm}\label{thm:boundforshortuv} Let $G$ be a group hyperbolic relative to the set of subgroups $\mathcal{P}=\{P_1,\dots,P_m\}$, and let $u$ and $v$ be two elements conjugate in $G$. Let $g\in G$ be a shortest conjugator for cyclic conjugates $u'$ and $v'$ of $u$ and $v$. Let $\alpha$ and $\beta$ be $(8\delta+1)$-local geodesics representing $u'$ and $v'$, correspondingly. If $\max\{l_{\alpha},l_{\beta}\}\leq 86\delta+3$ and $Q_g$ is not a parabolic quadrilateral then 
\begin{enumerate}
\item $g\in B(176\delta+8,C(3))$, or else
\item  $u$ is conjugate to $z_u\in B(4\delta, C(3))$ by an element $g_u\in B(94\delta+5,C(3))$ and $v$ is conjugate to $z_v\in B(4\delta, C(3))$ by an element $g_v\in B(94\delta+5,C(3))$.
\end{enumerate}
\end{thm}
\begin{proof} Consider the geodesic quadrilateral $Q_g$. The estimate for the relative length of $g$ in (1) and for the relative length of $g_u$ and of $g_v$ in (2), as well as the existence of $z_u$ and $z_v$ in (2), follow from Lemma~\ref{lem:linearbound}. The bound on the $\Gamma$-length of the parabolic components of $g$, $g_u$ and $g_v$ follows from Corollary~\ref{cor:72delta}, Lemma~\ref{lem:isolated} and Corollary~\ref{cor:conjugatorforparabolic} if the parabolic components of $\gamma_u$ and $\gamma_v$ are not connected, and from Corollary~\ref{cor:veryshortg} if parabolic components of $\gamma_u$ and $\gamma_v$ are connected in $Q_g$.
\end{proof}
\subsection{Upper bound on the length of a conjugating element} Suppose that $u\in F(S_i)$, for some $i$.  If $[u]_{P_i}\cap B_i\neq\emptyset$ (see section~\ref{sec:notation}) then we set 
$M_u=\min_{t\in[u]_{P_i}\cap B_i} \{|y|_{\Gamma}\mid y\in P_i,\ u=yty^{-1} \};$
otherwise, we set $M_u=0$. Note that if $u\in B_i$ then $M_u=0$ as well. We also set $M_u=0$ if $u\notin F(S_1)\cup\dots\cup F(S_m)$. 

Informally, $M_u\neq 0$ if and only if $u\in P_i$ for some $i$, $u\notin B_i$, so that $u$ itself is not ``very short'', whereas the conjugacy class $[u]_{P_i}$ of $u$ in $P_i$ contains ``very short'' elements. In this case, $M_u$ equals the length of a shortest $y\in P_i$ that conjugates $u$ into $B_i$.  We define $M_v$ similarly. 
\begin{thm}\label{thm:upperboundconst} Let $G$ be a finitely generated group hyperbolic relative to the set of subgroups $\mathcal{P}=\{P_1,\dots,P_m\}$, and let $S_1,S_2,\dots,S_m$ be finite generating sets for the parabolic subgroups, so that $P_i=\left\langle S_i\right\rangle\forall i $. Let $u$ and $v$ be conjugate in $G$, and let $g$ be a shortest conjugating element for cyclic permutations of $u$ and $v$.
\begin{enumerate}
\item If $u$ and $v$ are hyperbolic elements of $G$ then there is an upper bound on the $\Gamma$-length of $g$ independent of $u$ and $v$.
\item Let $u$ and $v$ be parabolic, and if $u\in F(S_i)$ for some $i$ then suppose that $v\notin[u]_{P_i}$. Then
$$|g|_{\Gamma}\leq\max\{M_u+M_v, 2(94\delta+5)C(3)\}+\# B(4\delta, C(3))(16\delta+2)C(3)+\frac 12K\sum\limits_{i=1}^{m} |S_i|^{C(3)},$$ where $M_u$ and $M_v$ are as above, and $K$ is as in Lemma~\ref{lem:conjugatorfor4delta}.
\end{enumerate}
\end{thm}
\begin{proof} If $\max\{l_{\alpha},l_{\beta}\}> 86\delta+3$ then the claim follows from Theorem~\ref{thm:conjugatorforlongelts}. Note that in this case, $u$ and $v$ are hyperbolic. 
If $u$ and $v$ are hyperbolic and short, that is, $\max\{l_{\alpha},l_{\beta}\}\leq 86\delta+3$, then $|g|_{\Gamma}\leq |g_u|_{\Gamma} +|g_v|_{\Gamma} +|f|_{\Gamma},$
where a constant bound on the length of $g_u,g_v$ is given in Theorem~\ref{thm:boundforshortuv}(2), and $f$ is as in Lemma~\ref{lem:conjugatorforhyperbolics}. This proves (1).

To prove (2), note that $|g|_{\Gamma}\leq\max\{M_u+M_v, |g_u|_{\Gamma} +|g_v|_{\Gamma}\}+|h|_{\Gamma}$, where $h$ is as in Lemma~\ref{lem:conjugatorfor4delta}, and a constant bound on the length of $g_u,g_v$ is given in Theorem~\ref{thm:boundforshortuv}(2). It depends on whether $u$ and $v$ are written as parabolic or hyperbolic words, what part of the estimate applies actually.
\end{proof}
\begin{rem}\label{rem:conjparabolic}
In general, if $u$ and $v$ are written as parabolic words in $G$ then no upper bound on the length of $g$ independent of $u$ and $v$ can be given, even if $u$ and $v$ are not conjugate in a parabolic subgroup, as we assume in Theorem~\ref{thm:upperboundconst}(2). Indeed, $M_u$ and $M_v$ depend on $u$ and $v$, unless both $u$ and $v$ are ``very short''. If $u$ and $v$ are conjugate in a parabolic subgroup, so that $v\in[u]_{P_i}$ for some $i$ then $g\in P_i$ and the length of $g$ is completely determined by the properties of $P_i$, which we have no control over. 

However, if the parabolic subgroups are abelian then $K_i=0, \forall i$ (see Definition~\ref{def:Ki}) and $M_u=M_v=0$. Therefore, the upper bound on the length of $g$ in Theorem~\ref{thm:upperboundconst}(2) does not depend on $u$ and $v$. Thus, we have a constant bound also in the case when $u$ and $v$ are parabolic.
\end{rem}

\section{Preliminary computations}\label{sec:Preliminary}
In this section we explain how we compute data used in our algorithms. Neither the data nor the computations depend on $u$ and $v$.

\subsection{Solution to the word problem in parabolic subgroups is given} 
In this section and in what follows we assume that solution(s) to the word problem in the parabolic subgroups of $G$ is part of the input of our algorithms.

 In the following theorem we collect some known results.
\begin{thm}\label{prop:precomputation}
Let $G$ be a finitely presented group hyperbolic relative to subgroups $P_1,P_2,\dots,P_m$. Given a finite presentation for $G$, generating sets $S_1,S_2,\dots,S_m$ for the parabolic subgroups and solution to the word problem in the parabolic subgroups, one can compute the following data.
\begin{enumerate}
\item A relative Dehn presentation for $G$ and a factor $N$ for a relative linear isoperimetric inequality for $G$ (Dahmani~\cite[Theorem 0.1]{DahmaniFindingRelHypStructures}, see also~\cite{DahmaniGuirardel}).
\item A hyperbolicity constant $\delta$ for the coned-off Cayley graph $\hat\Gamma$ of $G$.
\item Constants $C(2)$, $C(3)$ and $C(7,2\delta)$ (Osin ~\cite{OsinRelHypGps}).
\end{enumerate}
\end{thm}
\begin{proof} (2) Using the factor $N$ from (1), we can apply results from \cite{AlonsoEtAl}, \cite{Lysenok90}, \cite{Olshanskii91} or  \cite{Rebbechi} (see also~\cite[III.H Theorem 2.1]{BridsonHaefliger} for details) to compute a hyperbolicity constant $\delta$ for the coned-off Cayley graph $\hat\Gamma$ of $G$. 

(3) Given a relative Dehn presentation $G=\left\langle X,P_1,\dots,P_k\mid R=1,\ R\in\mathcal{R}\right\rangle, $ we find the maximum relative length $M=\max_{R\in\mathcal{R}} l_{\hat{\Gamma}}(R);$ note that $M\leq 8\delta$. Next, we can compute a bound $B=B(\delta,\lambda,\varepsilon,D)$ on the relative Hausdorff distance between two relative $(\lambda,\varepsilon)$-quasi-geodesics, whose endpoints are distance $D$ apart (see for instance,~\cite[Lemma 3.8]{OsinRelHypGps}). Now, the proof of \cite[Proposition 3.15]{OsinRelHypGps} shows that the $\Gamma$-length of an isolated parabolic component $p$ of a  relative $(\lambda,\varepsilon)$-quasi-geodesic is bounded as follows:
$l_{\Gamma}(p)\leq (8\lambda B+2\varepsilon+2B)MK.$ Here, and only here, in this proof, $K$ denotes the constant from \cite[Lemma 3.8]{OsinRelHypGps}.

The paths that we consider are $(7,2\delta)$-, $(3,0)$-, or $(2,0)$-quasi-geodesics. Therefore, we are interested in the following  upper bounds on the length of an isolated parabolic component: 
\begin{align*}
C(7,2\delta) &=(56 B_7+4\delta+2B_7)MK,\ \text{where }\ B_7=B(\delta,7,2\delta,0), \\
C(3) &=C(3,0)=24B_3MK,\ \text{where}\ B_3=B(\delta,3,0,0),\\ 
C(2) &=C(2,0) =18B_{2}MK, \ \text{where}\ B_{2}=B(\delta,2,0,0).
\end{align*}
\end{proof}
We use the constants from Theorem~\ref{prop:precomputation}(3) and the solution to the word problem in parabolic subgroups to compute the following lists of parabolic elements:
\begin{align*}
\mathcal{L}_1 &= \{p\in F(S_1)\cup F(S_2)\cup\dots\cup F(S_m)\mid |p|_{\Gamma}\leq C(2)\} \\
\mathcal{L}_2 &= \{q\in F(S_1)\cup F(S_2)\cup\dots\cup F(S_m)\mid |q|_{\Gamma}\leq C(7,2\delta)\} \\
\mathcal{L}_3 &=\{p\in F(S_1)\cup F(S_2)\cup\dots\cup F(S_m)\mid |p|_{\Gamma}\leq C(3)\}.
\end{align*}
Recall that $B(r_1,r_2)$ denotes the set of those elements of $G$ whose relative length does not exceed $r_1$ and the length of each parabolic component is bounded above by $r_2$ (see section~\ref{sec:notation}). Clearly, there are finitely many such elements and we can effectively enumerate all of them. So, we use the solution to the word problem in $G$ and the lists $\mathcal{L}_1$ and $\mathcal{L}_2$ to compute the following subsets of the Cayley graph $\Gamma$ of $G$:
\begin{align*}
\mathcal{L}_4 &= B(7\delta+1,C(7,2\delta)) \\
\mathcal{L}_5 &= B(16\delta+1,C(2)) \\
\mathcal{L}_6 &= B(2(274\delta+9),C(7,2\delta))\\
\mathcal{L}_8 &=\{w\in F(S)\mid w\in B(86\delta+3,2C(3)),\ \text{and}\ w\ \text{is a relative geodesic}\};\\
\mathcal{L}_9 &=\{w\in F(S)\mid |w|_{\Gamma}\leq 4\delta C(3)\}.
\end{align*}
% % % % %

\subsection{Solution to the conjugacy problem in parabolic subgroups is given} 
\begin{prop}\label{parabolicconjclasses} Let $G$ be a group hyperbolic relative to $\mathcal{P}=\{P_1,\dots,P_m\}$, and let $S=S_0\cup S_1\cup\dots\cup S_m$ be a generating set for $G$ such that $P_i=\left\langle S_i\right\rangle $, for all $i=1,2,\dots,m$. Suppose that solutions to the word problem and to the conjugacy problem in each one of the parabolic subgroups are given. Then there is an algorithm to compute the bounded conjugacy classes $$\mathcal{BCC}=\{[w]_{B(4\delta,C(3))}\mid w\in B(4\delta,C(3)) \}.$$ In particular, we can compute the list $\mathcal{L}_{11}\subset\mathcal{L}_3\times\mathcal{L}_3$ of all the pairs $(p,q)$ of conjugate elements $p,q\in\mathcal{L}_3$ and a set of conjugating elements $\mathcal{L}_7=\{g_{pq}\in G\mid q=g_{pq}pg_{pq}^{-1}, (p,q)\in \mathcal{L}_{11} \}$.
\end{prop}
\begin{proof} Let $B_i=P_i\cap\mathcal{L}_3$ be the set of the elements of length at most $C(3)$ in $P_i$. Firstly, using the soluton to the conjugacy problem in the parabolic subgroups, we compute the bounded conjugacy classes $\{[p]_{P_i}\cap B_i\mid p\in B_i, i=1,2,\dots,m \}$ of very short elements of $P_i$, for each $i$. We can also find a conjugating element  for each pair $q_1,q_2\in [p]_{P_i}\cap B_i$, $\forall p\in B_i, i=1,2,\dots,m$. This allows us to compute the constants $K_i$ from definition~\ref{def:Ki}. Now we can apply Lemma~\ref{lem:conjugatorfor4delta} to compute $[p]_{B(4\delta,C(3))}$ $\forall p\in B_i,\forall i$. More precisely, we conduct exhaustive search of all the elements in $[p]_{B(4\delta,C(3))}$ by taking all the elements $h\in B(K_{4\delta},K)$ (with the notation of Lemma~\ref{lem:conjugatorfor4delta}) and checking for each one of them whether or not $h^{-1}ph=w$ for some $w\in B(4\delta,C(3))$. In particular, if $w\in \mathcal{L}_3$ then we add the pair $(p,w)$ to $\mathcal{L}_{11}$ and $h$ to $\mathcal{L}_7$. According to Lemma~\ref{lem:isolated}(2) and Lemma~\ref{lem:noselfintersection}(2), if $w\in B(4\delta,C(3))$ is a parabolic element of $G$ then necessarily it is conjugate to some $p\in \mathcal{L}_3$. Therefore, the argument in the proof of Lemma~\ref{lem:conjugatorfor4delta} implies that the exhaustive search procedure described above provides all the conjugacy classes of parabolic elements in $\mathcal{BCC}$.

It remains to compute the bounded conjugacy classes $[w]_{B(4\delta,C(3))}$ of all the hyperbolic elements $$w\in B(4\delta,C(3))\setminus \bigcup_{p\in\mathcal{L}_3} [p]_{B(4\delta,C(3))}.$$ This can be done using the estimate from Lemma~\ref{lem:conjugatorforhyperbolics}.
% % % % % % %
%In fact, according to Corollary~\ref{cor:shortbasetower}, we only need to use the estimate from Lemma~\ref{lem:conjugatorforhyperbolics} to find pairs of elements from $B(4\delta,C(3))$ that are conjugate in $G$. 
% % % % % %

\end{proof}

For our solution to the conjugacy problem we also need to pre-compute the following sets:
\begin{align*}
% Lemma~\ref{lem:replacebygeodesics} 4.5
\mathcal{L}_{88} &=\{(w_1,w_2)\in \mathcal{L}_8\times\mathcal{L}_8\mid w_1,w_2\ \text{are conjugate in}\ G\}; \\
\mathcal{L}_{10} &=\{w\in \mathcal{L}_8\mid (w,q)\in \mathcal{L}_{88}\ \text{with}\ q\in\mathcal{L}_{3}  \}.
\end{align*}
For our solution to the conjugacy search problem, we also need to pre-compute the set $$\mathcal{L}_{12}=\{g_{wz}\in G\mid z=g_{wz}wg_{wz}^{-1},\  (w,z)\in \mathcal{L}_{88}\}$$ of conjugating elements for the pairs from $\mathcal{L}_{88}$.

We need the estimate from Lemma~\ref{lem:linearbound} and the set $\mathcal{BCC}$ from Proposition~\ref{parabolicconjclasses} to compute $\mathcal{L}_{88}$. Then it is straightforward to find $\mathcal{L}_{10}$ and $\mathcal{L}_{12}$.

\section{Algorithms}\label{sec:Alg}
\subsection{Curve shortening} 

The complexity of the algorithms below depends on the (maximum) complexity $O(C_w^{(par)}(n))$ of the solution to the word problem in the parabolic subgroups. Some algorithms use the solution to the conjugacy problem in the parabolic subgroups; we denote the (maximum) complexity of that by $O(C_c^{(par)}(n))$. Finally, the algorithm for the conjugacy search problem uses the solution to the conjugacy search problem in parabolic subgroups, the complexity of which we denote by $O(C_{search}^{(par)}(n))$.

The following lemma is proved in~\cite{Farb}.
\begin{lem}\label{lem:Farb} (Farb) Let $G$ be a relatively hyperbolic group defined by a Dehn presentation, and suppose that solution to the word problem in parabolic subgroups is given. There is an algorithm that takes as input a relative Dehn presentation $G=\left\langle S_0,P_1,\dots,P_k\mid R=1,\ R\in\mathcal{R}\right\rangle,$ finite  generating sets $S_1,S_2,\dots,S_m$ for the parabolic subgroups and a word in the generators $w_u\in F(S)$, where $S=\cup_{i=0}^m S_i$, and computes a relative $(8\delta+1)$-local geodesic $\rho$ such that $lab(\rho)=u$ in $G$; here $u$ is the element of $G$ defined by $w_u$.

If the complexity of the word problem in parabolic subgroups is $O(C_w^{(par)}(n))$ then the complexity of the algorithm is $O(C_w^{(par)}(\bar{L})\log \bar{L})$, where $\bar{L}=|w_u|_{S}$ is the length of the word $w_u$ in $F(S)$.
\end{lem}
\begin{proof} We give a sketch of the proof here. Let $k=8\delta+1$.

First, using the solution to the word problem in parabolic subgroups, we replace every maximal parabolic component of $w_u$ by a geodesic word in $\mathcal{L}_1$, whenever possible; we call $\bar{w}_u$ the word that we obtain in this way. A maximal parabolic component of $w_u$ is identified as a maximal subword in $F(S_i)$ for some $i$; the total length of the parabolic subwords of $w_u$ is bounded above by $\bar{L}=|w_u|_{S}$. Therefore, the complexity of this procedure is $O(C_w^{(par)}(\bar{L}))$.

We replace $\bar{w}_u$ with a relative $k$-local geodesic $\rho$, as follows. If the word $\bar{w}_u$ is not (the label of) a relative $k$-local geodesic then it has a subword $z$ of length at most $k$ so that every subword of $z$ is a relative geodesic but $z$ is not a relative geodesic. It follows that $z$ is a relative 2-quasi-geodesic. Let $y$ be a relative geodesic joining the endpoints of $z$. One shows that every parabolic component in the path $zy^{-1}$ is isolated. Therefore, $l_{\Gamma}(zy^{-1})\leq (k+k-1)C(2)=(16\delta+1)C(2)$, or $zy^{-1}\in\mathcal{L}_5$. So, to obtain a relative $(8\delta+1)$-local geodesic $\rho$, it suffices to replace every longer part of a word in $\mathcal{L}_5$ by its shorter part, whenever it occurs in $\bar{w}_u$. It may happen that after a replacement two or three parabolic components merge: for instance,  having replaced $x$ in a subword $xp$, we could obtain $wqp$ with $q$ and $p$ in the same parabolic subgroup. In this case we replace the new maximal parabolic component $qp$ by a geodesic  word in $\mathcal{L}_1$, if possible, and continue the computation. According to~\cite[Theorem 3.7]{Farb}, the computation time is $O(C_w^{(par)}(\bar
L)\log \bar{L})$, and the worst case scenario occurs when the path labelled by $w_u$ backtracks often, that is, it visits the same left coset of a parabolic subgroup many times.
\end{proof}
\begin{cor} \label{cor:lengthofrho}
With the notation of Lemma~\ref{lem:Farb}, the $\Gamma$-length of the relative $(8\delta+1)$-local geodesic $\rho$ can be bounded as follows: $l_{\Gamma}(\rho)<C(2)|w_u|_{S}$.
\end{cor}
\begin{proof}
Clearly, the relative length of $\rho$ does not exceed the relative length of the path labelled by $w_u$. The upper bound on the $\Gamma$-length of $\rho$ could be attained if $w_u$ had no parabolic components of length greater than one, while every subword of $\rho$ of relative length 1 was a parabolic word of $\Gamma$-length $C(2)$.
\end{proof}
\begin{figure}
\setlength{\unitlength}{.6cm}
\begin{picture}(16, 10)
	%FIRST COPY
		\qbezier(2, 8)(3, 6)(8, 2)					%curve
		\put(4, 5.5){\vector(1, -1){0}}				%arrowhead
		\put(2,8){\circle*{.2}}	                          % left end bold vertex
		\put(8,2){\circle*{.2}}	                          % right end bold vertex
		\put(2.35,7.4){\circle*{.2}}	                    % bold vertex
		\put(5.8,3.8){\circle*{.2}}	                    % bold vertex
		\put(1.6, 7.4){$\eta_i$}						%labeling
		\put(2.8, 5.6){$\bar{\rho}_i$}			   	%labeling
	    \put(6.7, 2.2){$\nu_i$}						%labeling
%
	%SECOND COPY
		\qbezier(8, 2)(12, 5)(15.5, 6)				%curve
		\put(10.9, 4){\vector(1, 1){0}}				%arrowhead
		\put(11.5, 3.8){$\rho_i$}						%labeling
		\put(8.6,2.4){\circle*{.2}}	                          % bold vertex
		\put(15.5,6){\circle*{.2}}	                          % right end bold vertex
%
   %SHORTENING
       	\qbezier(5.8, 3.8)(7, 3)(8.6, 2.4)					%curve
       	\put(7.4, 2.85){\vector(1, -1){0}}				%arrowhead
        \put(7.2, 3.3){$\bar{\nu}_i$}						%labeling
\end{picture}
	\caption{Cyclic curve shortening, see Proposition~\ref{cor:Lemma431}.}\label{fig:CurveShort}
\end{figure}
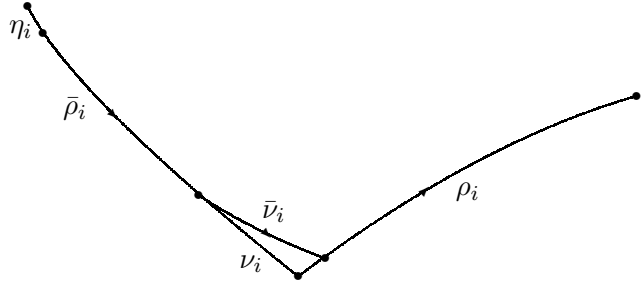
\begin{prop} \label{cor:Lemma431}  With the notation and assumptions of Lemma~\ref{lem:Farb}, there is an algorithm to compute a relative cyclic $(8\delta+1)$-local geodesic $\alpha$ such that $lab(\alpha)\in[u]_G$; the computation time is $O(\bar{L}C_w^{(par)}(\bar{L}))$. Moreover, the algorithm also finds $a\in G$ such that $lab(\alpha)=a^{-1}ua$, and the $\Gamma$-length of $a$ is bounded as follows: $|a|_{\Gamma}\leq C(2)|w_u|_{S}$.
\end{prop}
\begin{proof} If $u$ is not cyclically reduced, so that $u=a'u'a'^{-1}$ for some $a'\neq 1$, then we apply the free reduction; let the subpath $\rho_0$ of $\rho$ be such that $lab(\rho_0)=u'=a'^{-1}ua'$ is cyclically reduced.  To compute the relative cyclic $(8\delta+1)$-local geodesic $\alpha$ and a conjugating element $a$ such that $lab(\alpha)=a^{-1} lab(\rho) a$, we apply the procedure from Lemma~\ref{lem:Farb} to the concatenation $\rho_0\circ \rho_0$.

If the concatenation $\rho_0\circ \rho_0$ is not a relative $(8\delta+1)$-local geodesic then necessarily there is a 'tail' $\nu_0$ and a 'head' $\eta_0$ of $\rho_0$ such that the following conditions hold, see Figure~\ref{fig:CurveShort}:
\begin{enumerate}
\item $\rho_0=\eta_0 \circ\bar{\rho}_0\circ\nu_0$;
\item every proper subpath of $\nu_0 \circ\eta_0$ is a relative geodesic but $\nu_0 \circ\eta_0$ is not a relative geodesic;
\item the relative length of $\nu_0 \circ\eta_0$ does not exceed $8\delta+1$. 
\end{enumerate}
Since $\rho_0$ is a relative $(8\delta+1)$-local geodesic, both $\eta_0 $ and $\nu_0$ are nontrivial paths; in particular, the relative length of $\eta_0$ does not exceed $8\delta$.  We replace $\nu_0 \circ\eta_0$ with a relative geodesic $\bar{\nu}_0$. By an argument due to Farb, $\nu_0 \circ\eta_0$ and $\bar{\nu}_0$ form a pair of relative $(2,0)$ quasi-geodesics that do not have connected parabolic components (cf. the proof of Lemma~\ref{lem:Farb}). It follows that the length of every parabolic component of $\eta_0$ is bounded by $C(2)$. We set $a_0=lab(\eta_0)$, then $a_0\in B(8\delta,C(2))$, in particular, $|a_0|_{\Gamma}\leq 8\delta C(2)$.
 
Let $\rho'_1=\bar{\rho}_0\circ\bar{\nu}_0$. Note that $lab(\rho'_1)=a_0^{-1}lab(\rho_0) a_0=a_0^{-1}u'a_0=a_0^{-1}a'^{-1}ua'a_0$. We replace $\rho'_1$ with a relative $(8\delta+1)$-local geodesic $\rho_1$ with the same endpoints as $\rho'_1$ and continue with $\rho_1\circ\rho_1$ in a similar way. Namely, we identify $\nu_1$, and $\eta_1$ that satisfy the conditions (1)-(3) above, and then find $\bar{\nu}_1$ and $\rho_2$; note that $lab(\rho_2)=a_1^{-1}lab(\rho_1) a_1=a_1^{-1}a_0^{-1}u'a_0a_1$. After a number $s$ of iterations we obtain a relative $(8\delta+1)$-local geodesic path $\rho_s\circ\rho_s$ such that $lab(\rho_s)=a_s^{-1}\dots a_1^{-1}a_0^{-1}a'^{-1}ua'a_0a_1\dots a_s$. We set $\alpha=\rho_s$ and $a=a'a_0a_1\dots a_s$. The argument, that was applied above to the parabolic components of $\eta_0$ and the length of $a_0$, appllies to $\eta_i$ and $a_i$ for all $i$. Therefore, $|a_i|_{\Gamma}\leq 8\delta C(2),\forall i=0,1,\dots,s$. The number $s$ of iterations does not exceed $\bar{L}$ because $l_{\rho_0}\leq \bar{L}$ and $l_i\leq l'_i<l_{i-1}$, where $l_i=l_{\rho_i}$ and $l'_i=l_{\rho'_i}$, for all $i=1,2,\dots,z$. Note that $a'$ is a subword of $w_u$.
Thus, $|a|_{\Gamma}\leq 8\delta C(2)|w_u|_{S}=8\delta C(2)\bar{L}$.

To compute the time complexity, it remains to show that the complexity of the replacement of $\rho'_i$ with $\rho_i$ is bounded above by $O(C_w^{(par)}(n))$. Note that $\rho'_i$ is the concatenation of two relative $(8\delta+1)$-local geodesics. By Lemma~\ref{lem:localgeodnobacktracking}, relative $(8\delta+1)$-local geodesics do not backtrack. Therefore, in the curve shortening procedure, we may only need to replace each parabolic component of $\rho'_i$ with shorter parabolic components a globally bounded number of times. The total length $n_1+n_2+\dots+n_t$ of the parabolic components of $\rho'_i$ does not exceed $\bar{L}$, so that we have
\[
\Sigma_{i=1}^t O(C_w^{(par)}(n_i))\leq O(C_w^{(par)}(n_1+n_2+\dots+n_t))\leq O(C_w^{(par)}(\bar{L})).
\]
\end{proof}

\begin{thm}\label{thm:parabolic} Let $G$ be a group, hyperbolic relative to the set of subgroups $\mathcal{P}=\{P_1,\dots,P_m\}$. Let $S=\cup_{i=0}^{m} S_i$ be a finite generating set for $G$ such that $P_i=\left\langle S_i\right\rangle $ for $i=1,2,\dots,m$. Let $F(S)$ be the free group on $S$.  There is an algorithm that, given a relative Dehn presentation for $G$, solution to the word problem in the parabolic subgroups and a word $w_u\in F(S)$, decides whether or not the element $u\in G$, defined by the word $w_u$, is a hyperbolic or a parabolic element of $G$. Moreover, if $u$ is parabolic then the algorithm finds $q\in F(S_1)\cup\dots\cup F(S_m)$ conjugate to $u$.

The time complexity of the algorithm is $O(\bar{L}C_w^{(par)}(\bar{L}))$, where $\bar{L}=|w_u|_S$.
\end{thm}
\begin{proof} By Proposition~\ref{cor:Lemma431}, we can assume that $u$ is represented by a relative cyclic $(8\delta+1)$-local geodesic path $\alpha$; let $\bar{w}_u=lab(\alpha)$. If $\bar{w}_u$ is a word in $S_i$ for some $i$ then clearly, $u$ is parabolic in $G$. If $\bar{w}_u\notin F(S_i)$ for all $i$, then $u$ is parabolic only if it is conjugate to an element $q$ of one of the subgroups from $\mathcal{P}$. By Theorem~\ref{thm:conjugatorforlongelts}, $u$ is parabolic only if $\ell_{\alpha}\leq 86\delta+3$. So, assume that $\alpha$ is relatively short and does not consist of a single parabolic component. Let $\gamma_u$ be a relative geodesic joining the endpoints of $\alpha$, then by Lemma~\ref{lem:localgeodnobacktracking}, $\gamma_u$ cannot consist of a single parabolic component. Clearly, $|\gamma_u|_{\hat{\Gamma}}\leq 86\delta+3$. Consider the geodesic quadrilateral $Q_g$. By Lemma~\ref{lem:veryshortg}, if in $Q_g$ a parabolic component of $\gamma_u$ is connected to $q$ then necessarily $u=q$ in $G$. In this case, since $\gamma_u$ is a relative geodesic representing $u$, it consists of a single parabolic component, which is a contradiction. Therefore, no parabolic component of $\gamma_u$ is connected to $q$. Hence, we can apply Lemma~\ref{lem:isolated} and Lemma~\ref{lem:noselfintersection}(2) to show that $|q|_{\Gamma}\leq C(3)$, and the $\Gamma$-length of every parabolic component of a shortest conjugating element $g$ is bounded by $C(3)$. We conclude that the $\Gamma$-length of every parabolic component of $\gamma_u$ is bounded by $2C(3)$, so that $\gamma_u\in\mathcal{L}_{8}$. 

Thus, there is an algorithm to decide whether or not $w_u$ represents a parabolic element, as follows: 
\begin{enumerate}
\item Compute a relative cyclic $(8\delta+1)$-local geodesic path $\alpha$, using Lemma~\ref{lem:Farb} and Proposition~\ref{cor:Lemma431}; this takes $O(\bar{L}C_w^{(par)}(\bar{L}))$ steps. If $\ell_{\alpha}>86\delta+3$ then $u$ is not a parabolic element. 
\item If $\ell_{\alpha}\leq 86\delta+3$ but $\alpha$ does not consist of a single parabolic component then check whether there is $w\in \mathcal{L}_{10}$ such that $\bar{w}_u=w$. The given element $u$ is parabolic if and only if such $w$ exists. There are finitely many equalities to check, so that the complexity of this part of the algorithm is $O(C_w^{(par)}(\bar{L})\log \bar{L})$, the same as the complexity of the word problem in $G$.
\end{enumerate}
The complexity of the algorithm is $O(\bar{L}C_w^{(par)}(\bar{L}))$, because the two parts of it apply one after another.
\end{proof}

% % % % % % % % % % % %
\subsection{Conjugacy problem} 
\begin{convention} \label{convention51}
In this section, the following data is considered input of our algorithms.
\begin{enumerate}
\item[(i)] A relative Dehn presentation $\left\langle S_0,P_1,\dots,P_m \mid \mathcal{R} \right\rangle $ of a finitely generated relatively hyperbolic group $G$, along with finite generating sets $S_i$ for parabolic subgroups: $P_i=\left\langle S_i\right\rangle $ for $i=1,2,\dots,m$. We denote by $S$ the finite generating set $S=\cup_{i=0}^{m} S_i$ for $G$;
\item[(ii)] Solution(s) to the word problem in the parabolic subgroups; let $C_w^{(par)}(n)$ denote the (maximum) complexity of these procedures;
\item[(iii)] Two words $w_u, w_v\in F(S)$, where $F(S)$ is the free group on $S$. The maximum length $\bar{L}=\max\{|w_u|_S,|w_v|_S \}$ of these words is considered the length of the input. We denote by $u$ and $v$ the elements of $G$, defined by the words $w_u$ and $w_v$, respectively.
\item[(iv)] We assume that relative cyclic $(8\delta+1)$-local geodesic paths $\alpha$ and $\beta$ with $\bar{w}_u=lab(\alpha)$ and $\bar{w}_v=lab(\beta)$, such that $\bar{w}_u\in [u]_G$ and $\bar{w}_v\in [v]_G$, have already been computed; recall that this computation takes $O(\bar{L}C_w^{(par)}(\bar{L}))$, according to Proposition~\ref{cor:Lemma431}. We have the following inequality:  $L=\max\{l_{\alpha},l_{\beta} \}\leq \bar{L}$.
\end{enumerate}
\end{convention}

In the following theorem we describe an efficient algorithm to solve the conjugacy and the conjugacy search problem for ``long'' elements. Notably, solution to the conjugacy (search) problem in parabolic subgroups is not needed in this case. 
\begin{thm} \label{thm:alglongelts} There is an algorithm which takes as input all of the data listed in Convention~\ref{convention51}, checks the value of $L$ and proceeds as follows. If $L<86\delta+3$ then it stops. If $L\geq 86\delta+3$ then the algorithm decides whether or not $u$ and $v$ are conjugate in $G$. Moreover, if $u$ and $v$ are conjugate then a conjugating element for $\bar{w}_u$ and $\bar{w}_v$ will be found. The time complexity of the algorithm is $O(\bar{L}^2 C_w^{(par)}(\bar{L}) \log \bar{L})$.
\end{thm}
\begin{proof} If $\bar{w}_u$ and $\bar{w}_v$ are conjugate in $G$ then by Theorem~\ref{thm:conjugatorforlongelts}, there is a conjugating element $g\in\mathcal{L}_4$. The cardinality of $\mathcal{L}_4$ is a constant that does not depend on $u$ and $v$. Try every element $x\in\mathcal{L}_4$, with all the cyclic permutations of $\alpha$ and $\beta$. For each $x$ there are at most $L^2$ products $x\tilde{u}x^{-1}\tilde{v}^{-1}$, where $\tilde{u}=lab(\tilde{\alpha})$ and $\tilde{v}=lab(\tilde{\beta})$ are the labels of cyclic permutations of $\alpha$ and $\beta$, correspondingly. Furtermore, it takes $O(C_w^{(par)}(L)\log L)$ steps to decide whether or not $x\tilde{u}x^{-1}\tilde{v}^{-1}=1$. Since $L\leq\bar{L}$, the claim follows.
\end{proof}
% % 
The case when $u$ and $v$ are ``short'', which means that $L<86\delta+3$, is somewhat different. Note that we only assume that the relative length of $u$ and $v$ is bounded, while their $\Gamma$-length cannot be bounded in general. However, one can overcome this obstacle, as the following lemma shows.
\begin{lem}\label{lem:replacebygeodesics}
Assume that $L<86\delta+3$ and that $u$ and $v$ are hyperbolic elements of $G$. If $u$ and $v$ are conjugate in $G$ then one of the following holds.
\begin{enumerate}
\item There is a conjugating element $g\in\mathcal{L}_2$ for cyclic permutations of $\bar{w}_u$ and $\bar{w}_v$.
\item There are relative geodesics $\gamma_u$ and $\gamma_v$ with the same endpoints as $\alpha$ and $\beta$, correspondingly, such that $lab(\gamma_u),lab(\gamma_v)\in\mathcal{L}_8$.
\end{enumerate}
\end{lem}
\begin{proof} Let $\gamma_u$ (or $\gamma_v$) be a relative geodesic connecting the endpoints of $\alpha$ (or $\beta$). Clearly, 
$$\max \{|\gamma_u|_{\hat{\Gamma}},|\gamma_v|_{\hat{\Gamma}}\}\leq L< 86\delta+3.$$ 
By assumption, $\bar{w}_u$ and $\bar{w}_u$ are conjugate, so we consider the corresponding geodesic quadrilateral $Q_g$. If in $Q_g$ parabolic components of $\gamma_u$ and $\gamma_v$ are connected then it follows from Lemma~\ref{lem:veryshortg} that cyclic permutations of $\bar{w}_u$ and $\bar{w}_u$ are conjugate by some $g\in \mathcal{L}_2$. If no parabolic components of $\gamma_u$ and $\gamma_v$ are connected then it follows from Lemma~\ref{lem:isolated}(1) and Lemma~\ref{lem:noselfintersection}(2) that the $\Gamma$-length of every parabolic component of $g$ is bounded by $C(3)$. We conclude that the $\Gamma$-length of every parabolic component of $\gamma_u$ and of $\gamma_v$ is bounded by $2C(3)$, hence $\gamma_u,\gamma_v\in\mathcal{L}_8$.
\end{proof}

\begin{thm} \label{thm:algorithm} There is an algorithm that takes as input all of the data listed in Convention~\ref{convention51}, decides for each one of $u$ and $v$ whether the element is hyperbolic or  parabolic, and then does the following. 
\begin{enumerate}
\item If one of $u$ and $v$ is hyperbolic and the other one is parabolic in $G$ then the algorithm stops. Clearly, $u$ and $v$ are not conjugate. The time complexity of the procedure in this case is $O(\bar{L}C^{(par)}_w(\bar{L}))$.
\item If both $u$ and $v$ are hyperbolic then the algorithm decides, whether or not $u$ and $v$ are conjugate. Moreover, if $u$ and $v$ are conjugate then the algorithm finds a conjugating element for some representatives of the conjugacy classes $[u]_G$ and $[v]_G$. The time complexity of the procedure in this case is $O(\bar{L}^2C^{(par)}_w(\bar{L})\log \bar{L})$.
\item If both $u$ and $v$ are parabolic then the algorithm needs also solution to the conjugacy problem for each one of the parabolic subgroups of $G$ to decide, whether or not $u$ and $v$ are conjugate. If $O(C^{(par)}_c(n))$ is the (maximum) time complexity of solution to the conjugacy problem in a parabolic subgroup of $G$ then the time complexity of this procedure is $\max\{O(C^{(par)}_c(\bar{L})),O(\bar{L}C^{(par)}_w(\bar{L}))\}$, 
\end{enumerate}
\end{thm}
\begin{proof} 
By Theorem~\ref{thm:parabolic}, we can decide whether each one of $u$ and $v$ is a hyperbolic or a parabolic element. If the results differ then $u$ and $v$ are not conjugate in $G$, and we have (1). Otherwise, we proceed as follows.

(2) If $L\geq 86\delta+3$ then we apply the algorithm from Theorem~\ref{thm:alglongelts}. If $L< 86\delta+3$ and $\bar{w}_u$ and $\bar{w}_v$ are not conjugate by $g\in \mathcal{L}_2$, then by Lemma~\ref{lem:replacebygeodesics}, for $\bar{w}_u$ and $\bar{w}_v$ be conjugate, the following conditions have to hold. There exist two relative geodesics, $\gamma_u$ and $\gamma_v$, such that their labels $\bar{z}_u=lab(\gamma_u)$ and $\bar{z}_v=lab(\gamma_v)$ are in $\mathcal{L}_8$, and we have that $\bar{z}_u=\bar{w}_u$ and $\bar{z}_v=\bar{w}_v$ in $G$. Moreover, the pair $(\bar{z}_u,\bar{z}_v)$ has to be in $\mathcal{L}_{88}$; if this is the case then there is a conjugating element for $\bar{z}_u$ and $\bar{z}_v$ in $\mathcal{L}_{12}$. If at least one of these conditions fails then $u$ and $v$ are not conjugate in $G$.

One checks in time $O(\bar{L}^2C^{(par)}_w(\bar{L})\log \bar{L})$ whether or not $\bar{w}_u$ and $\bar{w}_v$ are conjugate by some $g\in \mathcal{L}_2$, the complexity is the same as that from Theorem~\ref{thm:alglongelts}. If there is no conjugating element $g\in\mathcal{L}_2$, that is, the case of Lemma~\ref{lem:replacebygeodesics}(2) occurs, then  one only needs to check finitely many equalities, using the results of preliminary computations.

(3) By Theorem~\ref{thm:parabolic}, we have $q_u,q_v\in F(S_1)\cup F(S_2)\cup\dots\cup F(S_m)$ such that $q_u\in[u]_G$ and $q_v\in[v]_G$; possibly, $q_u=\bar{w}_u$ or $q_v=\bar{w}_v$. If $q_u,q_v\in F(S_i)$ for some $i$ then we use the solution to the conjugacy problem in $P_i$ to determine whether they are conjugate in $P_i$. If $[q_u]_{P_i}=[q_v]_{P_i}$ then $u$ and $v$ are conjugate in $G$. If $q_u$ and $q_v$ are not conjugate in $P_i$ then it is still possible that they are conjugate by an element in $G\setminus P_i$. We proceed as follows. Let $q_u\in P_i$ and $q_v\in P_j$ (possibly, $i=j$). We try to construct a semi-parabolic geodesic quadrilateral $Q_g$ with $q_u$ and $q_v$ as its top and bottom horizontal sides. By Lemma~\ref{lem:conjugateparabolic}, if such $Q_g$ exists then $[q_u]_{P_i}\cap\mathcal{L}_3\neq\emptyset$ and $[q_v]_{P_j}\cap\mathcal{L}_3\neq\emptyset$; otherwise, $u$ and $v$ are not conjugate in $G$. So, suppose that there are $p_u\in [q_u]_{P_i}\cap\mathcal{L}_3$ and $p_v\in [q_v]_{P_j}\cap\mathcal{L}_3$. By Proposition~\ref{parabolicconjclasses}, $u$ and $v$ are conjugate if and only if $(p_u,p_v)\in \mathcal{L}_{11}$. Note that conjugating elements for the pairs in $\mathcal{L}_{11}$ are collected in $\mathcal{L}_{7}$.

Using the solution to conjugacy problem in the parabolic subgroups $\# B_i$ and $\# B_j$ times (recall that $B_i=P_i\cap \mathcal{L}_3,\forall i$), it will take $O(C^{(par)}_c(\bar{L}))$ steps to find $p_u\in [q_u]_{P_i}\cap\mathcal{L}_3$ and $p_v\in [q_v]_{P_j}\cap\mathcal{L}_3$, or to make sure that at least one of the intersections is empty. The complexity of the algorithm from Theorem~\ref{thm:parabolic} is $O(\bar{L}C^{(par)}_w(\bar{L}))$. The other  procedures in this case have either the same or lower complexity.
\end{proof}

\subsection{Conjugacy Search Problem} 
Recall that in a countable group with solvable word problem, the conjugacy search problem is always solvable: enumerate all the elements of $G$ and substitute them one after another into the equation $xux^{-1}=v$. Since a solution exists, it will be found. This is why in the Theorem~\ref{thm:Search} below we do not assume that solution to the conjugacy search problem in the parabolic subgroups of $G$ is given. However, the estimate for the time complexity of our algorithm refers to the time complexity of the conjugacy search problem in parabolic subgroups. Indeed, in some cases a better algorithm for the parabolic subgroups may exist. For instance, if a parabolic subgroup $P$ is abelian then the conjugacy search problem in it can be solved instantly, because $[p]_P=[q]_P\Leftrightarrow p=q$ and any $g\in P$ is a conjugating element.  We denote the (best possible) ``parabolic'' complexity by  $O(C^{(par)}_{search}(n))$. 
On the other hand, our algorithm for the conjugacy search problem in $G$ uses solution to the conjugacy problem in parabolic subgroups if the given elements are parabolic, see Theorem~\ref{thm:Search}(2).
% % % % % %
\begin{thm} \label{thm:Search} There is an algorithm which takes as input all of the data listed in Convention~\ref{convention51} and the information that $u$ and $v$ are conjugate in $G$, decides whether $u$ and $v$ are hyperbolic or parabolic, and then does the following. 
\begin{enumerate}
\item If $u$ and $v$ are hyperbolic elements of $G$ then the algorithm finds a conjugating element for $u$ and $v$. The time complexity of the algorithm in this case is $O(\bar{L}^2C^{(par)}_w(\bar{L})\log \bar{L})$.
\item If $u$ and $v$ are parabolic elements of $G$ then the algorithm needs also solution(s) to the conjugacy problem for each one of the parabolic subgroups of $G$ to find a conjugating element for $u$ and $v$. If $O(C^{(par)}_c(n))$ is the (maximum) time complexity of these solutions then the time complexity of the algorithm is $\max\{O(C^{(par)}_c(\bar{L})),O(C^{(par)}_{search}(\bar{L})),O(\bar{L}C^{(par)}_w(\bar{L}))\}$.
\end{enumerate}
\end{thm}
\begin{proof} An algorithm to determine whether a given element $x\in G$ is hyperbolic or parabolic is described in Theorem~\ref{thm:parabolic}.

(1) Recall that by Convention~\ref{convention51}(iv), we assume that we have already computed relative cyclic $(8\delta+1)$-local geodesics $\alpha$ and $\beta$ with the labels $\bar{w}_u=lab(\alpha)$ and $\bar{w}_v=lab(\beta)$, such that $\bar{w}_u\in [u]_G$ and $\bar{w}_v\in [v]_G$. By the proof of Proposition~\ref{cor:Lemma431}, the procedure provides elements $a_u$ and $a_v$ such that $u=a_u \bar{w}_u a_u^{-1}$ and $v=a_v \bar{w}_v a_v^{-1}$. The algorithm described in the proof of Theorem~\ref{thm:algorithm}(2) provides an element $g\in G$ such that $\bar{w}_v=g\bar{w}_ug^{-1}$. We conclude that $v=(a_v g a_u^{-1}) u(a_v g a_u^{-1})^{-1}$. The complexity in this case is $O(\bar{L}^2C^{(par)}_w(\bar{L})\log \bar{L})$, that of the algorithm from Theorem~\ref{thm:algorithm}(2); the complexity of the algorithm from  Proposition~\ref{cor:Lemma431} is lower.

(2) By Theorem~\ref{thm:parabolic}, $u$ and $v$ are conjugate to parabolic elements $q_u$ and $q_v$, correspondingly; possibly, $\bar{w}_u=q_u$ and $\bar{w}_v=q_v$. If the latter equalities fail, which means that  $\bar{w}_u$ and $\bar{w}_v$ are written as hyperbolic words, then conjugating elements for $\bar{w}_u$ and $q_u$ as well as for $\bar{w}_v$ and $q_v$ can be found using the algorithm from  Theorem~\ref{thm:parabolic}.

Now, if $q_u,q_v\in F(S_i)$ for some $i$ then the solution to the conjugacy problem in $P_i$ allows us to determine whether or not $q_u$ and $q_v$ are conjugate in $P_i$. If they are conjugate then we can find a conjugating element for $q_u$ and $q_v$ in $P_i$, using an available algorithm for the conjugacy search problem in $P_i$. If $q_u\in P_i$ and $q_v\in P_j$ are not conjugate in a parabolic subgroup then, by Lemma~\ref{lem:conjugateparabolic}, there are elements $p_u\in [q_u]\cap B_i$ and $p_v\in [q_v]\cap B_j$. Using the solution to the conjugacy problem (finitely many times) and the solution to the conjugacy search problem in the parabolic subgroups, we find $p_u$, $p_v$ and corresponding conjugating elements. Whereas the elements $p_u$ and $p_v$ may not be unique, $(p_u,p_v)\in \mathcal{L}_{11}$ and a conjugating element for $p_u$ and $p_v$ is in $\mathcal{L}_7$ in any case, according to Proposition~\ref{parabolicconjclasses}.

We apply the algorithm from Theorem~\ref{thm:parabolic} and then we use solutions to the conjugacy problem and to the conjugacy search problem in parabolic subgroups. Each algorithm is applied finitely many times; the estimate for the complexity follows.
\end{proof}
% % % % % % % % % % % % % % % % % % % % % % % % % % %
The following theorem is immediate from Theorem~\ref{thm:Search}.
\begin{thm}\label{thm:conjsearch} 
There is an algorithm which takes as input all of the data listed in Convention~\ref{convention51}, solution to the conjugacy problem in the parabolic subgroups and the information that $u$ and $v$ are conjugate in $G$, and finds a conjugating element for $u$ and $v$.
If in parabolic subgroups the word problem can be solved in time $O(C^{(par)}_w(n))$, the conjugacy problem can be solved in time $O(C^{(par)}_c(n))$, and the conjugacy search problem can be solved in time $O(C^{(par)}_{search}(n))$ then the time complexity of the algorithm is 
%\newline 
$T_{search}(\bar{L})=\max\{O(C^{(par)}_{c}(\bar{L})),O(C^{(par)}_{search}(\bar{L})),O(\bar{L}^2C^{(par)}_w(\bar{L})\log \bar{L})\}$.
\end{thm}

As another application, we have the following generalization of Proposition~\ref{parabolicconjclasses}.

\begin{thm}\label{thm:boundedconjclasses} Let $G$ be a group hyperbolic relative to $\mathcal{P}=\{P_1,\dots,P_m\}$, and let $S=S_0\cup S_1\cup\dots\cup S_m$ be a generating set for $G$ such that $P_i=\left\langle S_i\right\rangle $, for all $i=1,2,\dots,m$. Suppose solutions to the word problem and to the conjugacy problem in each one of the parabolic subgroups are given. Then there is an algorithm that, given $u\in F(S)$ and a positive integer $N$, computes the ball $B_N$ of radius $N$ in the Cayley graph $\Gamma=\Gamma(G;S)$ and the bounded conjugacy class $[u]_G\cap B_N$. Moreover, the algorithm computes a conjugating element for each pair $x,y\in [u]_G\cap B_N$.
\end{thm}
\begin{proof} By \cite[Theorem 3.7]{Farb} (see Theorem~\ref{thm:Farb}), $G$ has solvable word problem, so that the ball $B_N$ can be computed. Since the ball $B_N$ is finite, the bounded conjugacy class $[u]_G\cap B_N$ can be computed using the algorithm from Theorem~\ref{thm:algorithm}. 
Conjugating elements can be computed by Theorem~\ref{thm:conjsearch}.
\end{proof}

%%%%%%%%%%%%%%%%%%%%%%%%%%%%%%%%%%%%%%%%%%%%%%%%%%%%%%%%%%%%%%%%%%%%%%%%
%%                  BIBLIOGRAPHY
%%%%%%%%%%%%%%%%%%%%%%%%%%%%%%%%%%%%%%%%%%%%%%%%%%%%%%%%%%%%%%%%%%%%%%%%
\bibliographystyle{plain}  
%\addcontentsline{toc}{chapter}{Reference}
%\begin{thebibliography}{99}
\bibliographystyle{plain}

%\bibliography
\bibliography{References}
%\bibliography{C:\Documents and Settings\inna\My Documents\preprints\ZnfreeCAT0\Znfree}
%\begin{thebibliography}{10}

%\end{thebibliography}

\end{document}